\PassOptionsToPackage{pagebackref}{hyperref}

\documentclass[14pt]{article}

\usepackage[a4paper,margin=1.03in]{geometry}
\usepackage{amsmath,amssymb,amsthm,mathtools,bm,mathrsfs}
\usepackage{enumitem}
\usepackage[expansion=false]{microtype}
\usepackage[colorlinks=true,linkcolor=blue,citecolor=blue,urlcolor=blue]{hyperref}
\usepackage[nameinlink,noabbrev,capitalize]{cleveref}

\allowdisplaybreaks
\numberwithin{equation}{section}
\newtheorem{theorem}{Theorem}[section]
\newtheorem{proposition}[theorem]{Proposition}
\newtheorem{lemma}[theorem]{Lemma}

\theoremstyle{definition}
\newtheorem{definition}[theorem]{Definition}
\theoremstyle{remark}

\newcommand{\R}{\mathbb R}

\newcommand{\eps}{\varepsilon}
\newcommand{\Id}{I}
\newcommand{\tr}{\operatorname{tr}}
\newcommand{\diver}{\operatorname{div}}
\newcommand{\diag}{\operatorname{diag}}
\newcommand{\dist}{\operatorname{dist}}
\newcommand{\cL}{\mathcal L}
\newcommand{\cJ}{\mathcal J}
\newcommand{\cK}{\mathcal K}
\newcommand{\cH}{\mathcal H}
\newcommand{\cA}{\mathcal A}

\author{ Zhihui Zhang}

\hypersetup{
	pdftitle={Exterior Serrin Rigidity for the Homogeneous k-Hessian Equation },
	pdfauthor={ Zhihui Zhang},
	pdfsubject={Exterior overdetermined problems for the homogeneous k-Hessian equation},
	pdfkeywords={k-Hessian equation, Serrin rigidity, exterior domain, Newton tensor, star-shaped domain}
}

\begin{document}
	\title{Exterior Serrin Rigidity for the Homogeneous
		\texorpdfstring{$k$}{k}-Hessian Equation\\}

    \date{}
         \maketitle
	
	\begin{abstract}
In this paper, we study the  exterior overdetermined problems for the homogeneous k-Hessian equations 
$$\sigma_k(D^2u)=0\quad\text{in}~\mathbb{R}^n\setminus\overline{\Omega}$$
in three dimensional regimes.
		For $2\le k<n$ and smooth strictly
		star-shaped domain $\Omega\Subset\R^n$,  we establish ball rigidity results
        %for exterior
		%overdetermined problems for the homogeneous $k$-Hessian equation 
        in all
		three dimensional regimes.  If $2\leq k<\frac{n}{2}$, the solution with
		boundary value $-1$ and limit zero at infinity is treated under strict
		$(k-1)$-convexity of $\partial\Omega$.  If $k=\frac{n}{2}$, the solution with boundary value $0$ and 
        logarithmic growth at infinity
		 is considered under strict $(k-1)$-convexity of $\partial\Omega$; if $k>\frac{n}{2}$, the solution with boundary $1$ and fundamental power-growth at infinity is considered under strict $(k-1)$-convexity of $\partial\Omega$.
		In each case, for a positive constant $c$,   $|Du|=c>0$ on $\partial\Omega$
        %constancy of the boundary gradient 
        forces $\Omega$ to be
		a Euclidean ball and determines the solution explicitly.
		
		The three arguments are organized by the single parameter
		$q=(n-k)/k$ and the scale function
		\[
		F_q(t)=\frac{t^{1-q}-1}{1-q},\qquad F_1(t)=\log t.
		\]
		After normalization, $U=F_q(v)$ and
		$M_q[v]=vD^2v-qDv\otimes Dv$ satisfies $\sigma_k(M_q[v])=0$.
		A unified contact-point calculation gives the sharp bound $|Dv|\le b$;
		a boundary viscosity contact then yields
		$H_k\ge qbH_{k-1}$.  The integral mechanism that closes the argument
		changes at the critical exponent: Rellich--Pohozaev identities are used
		for $q>1$, a scale-invariant Wronskian--Newton current for $q=1$, and a
		renormalized Newton--Jacobi mass for $q<1$.  When $q\leq 1,$ we also 
        obtain a global strict k-convex defining function by Minkowski gauge, this help us extend the exterior construction
		from convex domains to strictly star-shaped domains.

      %  show that the
		%Minkowski gauge supplies the global strict $k$-convex defining function
		%needed to extend the critical and supercritical exterior construction
		%from convex to strictly star-shaped domains.
	\end{abstract}
	
	
	
	\section{Introduction}\label{sec:introduction}
    In this paper, we consider the  exterior overdetermined problems for the homogeneous k-Hessian equations 
    \begin{equation}\label{eq:intro-equation}
        \sigma_k(D^2u)=0\quad
        %\text{in}~\mathbb{R}^n\setminus\overline{\Omega}
    \end{equation}
for $2\leq k<n.$
	% Over recent decades, there have been many results concerning the following equation 
	%\begin{equation}\label{eq:intro-equation}
	%	\sigma_k(D^2u)=0.
	%\end{equation}
    For $\lambda=(\lambda_1,...,\lambda_n)\in\mathbb{R}^n$, $1\leq k\leq n,$   the k-th elementary symmetric function is defined by
$$\sigma_j(\lambda)=\sum_{1\leq i_1<...<i_j\leq n}\lambda_{i_1}\dots\lambda_{i_j},\quad \sigma_0=1.$$
The G\aa rding cone is
\[
 \Gamma_k
 =\bigl\{\lambda\in\R^n:\sigma_j(\lambda)>0,
        \ 1\le j\le k\bigr\},
\]
and $\overline{\Gamma}_k$ denotes its closure.
Let $u$ be a $C^2$ function, the k-Hessian
operator is defined by
$$\sigma_k(D^2u)=\sigma_k(\lambda(D^2u)).$$
	%For a symmetric matrix $A$, let $\sigma_k(A)$ denote the $k$th
%	elementary symmetric function of its eigenvalues,
The equation
	\eqref{eq:intro-equation}
	is understood on the $k$-admissible branch.  The basic ellipticity and
	Dirichlet theory for Hessian equations originates in the work of
	Caffarelli, Nirenberg, and Spruck \cite{CNS}. It was subsequently extended to nonsmooth \(k\)-convex functions and Hessian measures by Trudinger and Wang \cite{TrudingerWangI,TrudingerWangII,
		TrudingerWangIII,Wang}.
     The radial fundamental solutions of
	\eqref{eq:intro-equation} change character at the threshold $n=2k$:
	they decay when $2k<n$, are logarithmic when $2k=n$, and have positive
	power growth when $2k>n$.  This trichotomy is the reason that exterior
	problems for \eqref{eq:intro-equation} require three different
	normalizations at infinity.

The rigidity question considered here belongs to the classical theory of
overdetermined boundary value problems. 
In a celebrated paper \cite{Serrin}, Serrin establishes the symmetry of the solution to
the Poisson problem
\begin{equation}\label{poisson}
		  \begin{cases}
		  		\Delta u=n&\text{in }~\Omega,\\
		  	u=0& \text{on}~\partial\Omega,\\
		  	\frac{\partial u}{\partial\nu}=1 &\text{on }~\partial\Omega,
		  \end{cases}
		\end{equation}	
where $\Omega$ is a smooth bounded,
open, connected domain, $\nu$ denotes the unit outward normal of $\partial\Omega.$ Through the method of moving planes and a refinement of the maximum principle,
they prove that if $u\in C^2(\overline{\Omega})$ is a solution of \eqref{poisson}, then $u=\frac{|x|^2-1}{2}$ up to translation and $\Omega$ is a ball.
%Alexandrov's reflection principle
%for hypersurfaces \cite{Alexandrov} was adapted by Serrin
%\cite{Serrin}, through the method of moving %planes, to show that a
%bounded domain supporting a solution of the classical overdetermined
%elliptic Dirichlet problem must be a ball. 
Weinberger
\cite{Weinberger} soon gave an alternative proof based on the maximum
principle for a suitable \(P\)-function together with a Pohozaev-type
integral identity. 
Brandolini, Nitsch, Salani and Trombetti \cite{BrandoliniNitschSalaniTrombetti} proposed an alternative proof of Serrin’s theorem. Through an integral-geometric argument, they also obtained the rigidity results for over-determined k-Hessian equations
\begin{equation}
		  \begin{cases}
		  		\sigma_k(D^2 u)=\binom{n}{k}&\text{in }~\Omega,\\
		  	u=0& \text{on}~\partial\Omega,\\
		  	\frac{\partial u}{\partial\nu}=1 &\text{on }~\partial\Omega,
		  \end{cases}
		\end{equation}	 
where $\Omega$ is a bounded domian of $\mathbb{R}^n $ and $\nu$ denotes the unit outward normal of $\partial\Omega.$ 
%The lower endpoint \(k=1\) of the \(k\)-Hessian family belongs to this
%classical theory, since

Further symmetry results for Hessian overdetermined problems, including mirror-symmetry and ring-domain results, were obtained by Wang and Bao \cite{WangBaoMirror,WangBaoRing}. More recently, Wang and Wang \cite{wangwang} studied a different exterior \(k\)-Hessian overdetermined problem with a positive constant right-hand side and quadratic asymptotics. In the isotropic case the corresponding domain is a ball, whereas anisotropic asymptotic matrices may lead to nonsymmetric domains. These results concern a nondegenerate quadratic-growth problem and are distinct from the homogeneous fundamental-growth setting considered here.

As for the exterior over-determined problems for the homogeneous k-Hessian equations considered in this paper, there are some results. When $k=1, $ equation \eqref{eq:intro-equation} reduces to

\[
   \Delta u=0.
\]
The corresponding exterior overdetermined problem is classical. 
In exterior domains, Reichel \cite{Reichel96,Reichel97} established the constant-gradient
characterization of balls for harmonic capacitary potentials by the method of 
moving planes. Garofalo and Sartori \cite{GarofaloSartori} extended the exterior symmetry theory to the \(p\)-capacitary setting under star-shapedness, while Poggesi  \cite{Poggesi} subsequently removed the a priori geometric assumption on the domain by combining a \(P\)-function, integral identities, the isoperimetric inequality, and a soap-bubble rigidity theorem. His critical case \(p=n\) covers, in particular, the planar logarithmic regime when \(p=n=2\), up to normalization.
Thus the endpoint \(k=1\) is by now well understood.
When $k=n,$ equation \eqref{eq:intro-equation} reduces to
%The opposite endpoint \(k=n\) behaves in a fundamentally different way.
%In this case
\[
   \det D^2u=0,
\]
 the equation becomes the homogeneous Monge--Amp\`ere equation on
the convex branch. The natural exterior Serrin rigidity statement is
false at this endpoint. Indeed, let
\(\Omega\subset\mathbb R^n\) be any smooth strictly convex domain, let
\[
    E:=\mathbb R^n\setminus\overline{\Omega},
    \qquad
    d_\Omega(x):=\operatorname{dist}(x,\overline{\Omega}),
\]
and, for any \(a>0\), define
\[
    u(x):=1+a\,d_\Omega(x).
\]
In outward normal coordinates \(x=y+t\nu(y)\), and in a principal frame
at \(y\in\partial\Omega\), one has
\[
    \lambda\bigl(D^2d_\Omega(x)\bigr)
    =
    \left(
        \frac{\kappa_1(y)}{1+t\kappa_1(y)},
        \ldots,
        \frac{\kappa_{n-1}(y)}{1+t\kappa_{n-1}(y)},
        0
    \right).
\]
Consequently, \(u\) is \(n\)-admissible and satisfies
\[
    \det D^2u=0
    \qquad\text{in }E.
\]
Moreover,
\[
    u=1,\qquad
    u_\nu=|Du|=a
    \qquad\text{on }\partial\Omega,
\]
and, since \(\Omega\) is bounded,
\[
    u(x)-a|x|=O(1)
    \qquad\text{as }|x|\to\infty.
\]
Thus every smooth strictly convex domian, and not only a ball, supports a
solution satisfying the natural linear-growth and constant-gradient
conditions. This does not conflict with Reichel's Monge--Amp\`ere
symmetry result, which assumes that the solution is uniformly convex on
every bounded subset \cite{Reichel96}. Such an assumption is
incompatible with the homogeneous equation \(\det D^2u=0\).
Accordingly, we consider the case $2\leq k<n$ in this paper.

	For the exterior Dirichlet problem of the homogeneous $k$-Hessian equation, Ma and Zhang \cite{MaZhang} considered the following k-Hessian equations
     \begin{equation}
		  \begin{cases}
		  		\sigma_k(D^2u)=0&\text{in }~\R^n\setminus\overline\Omega,\\
		  	u=-1~\text{if}~1\leq k<\frac{n}{2}, u=0~\text{if}~k=\frac{n}{2},~u=1~\text{if}~\frac{n}{2}<k& \text{on}~\partial\Omega,\\
		  	u(x)\to0~\text{if}~1\leq k<\frac{n}{2} ,~u(x)=\log|x|+O(1)~\text{if}~k=\frac{n}{2},~u(x)=|x|^{\frac{2k-n}{k}}+O(1)~\text{if}~\frac{n}{2}<k&\text{as }|x|\to\infty,
		  \end{cases}
		\end{equation}	
        where $\Omega$ is a smoothly convex domain in $\mathbb{R}^n$ and is strictly (k-1)-convex.
    They established the
	exterior Dirichlet theory, including existence, uniqueness,
	$C^{1,1}$ estimates, asymptotics in all three dimensional regimes, together with weighted boundary inequalities in the subcritical and critical regimes.
    %Their analysis covers the decaying, logarithmic, and
	%power-growth models associated with the three signs of $n-2k$.  
    For $k<\frac{n}{2}$, Xiao \cite{Xiao} constructed the exterior solution 
    %for the following equation
   % \begin{equation}
		%  \begin{cases}
		  		%\sigma_k(D^2u)=0&\text{in }~\R^n\setminus\overline\Omega,\\
		%  	u=-1& \text{on}~\partial\Omega,\\
	%	  	u(x)\to0\ &\text{as }|x|\to\infty,
		%  \end{cases}
	%	\end{equation}	
    on smooth
	 $(k-1)$-convex star-shaped domain $\Omega$ 
    and derived a generalized
	Minkowski inequality from a monotonicity formula.  Both works study weighted curvature integrals on level
	sets of the solution.  Building on these developments, Yin and Zhou \cite{YinZhou}
	introduced a more general monotone quantity in the subcritical regime
	and obtained geometric inequalities and an exterior overdetermined
	ball characterization under convexity. They also observed that the available Weinberger-type \(P\)-function does not appear, by itself, to close the fully nonlinear rigidity argument. 
	
	The preceding results leave two related issues.  First, 
	overdetermined theory should reflect three cases,
	whereas the existing ball characterization is subcritical.  Second,
	ordinary convexity is stronger than the star-shaped curvature
	hypotheses under which the exterior solution and the relevant
	Minkowski formulas naturally live.  The purpose of this paper is to
	resolve these issues in a unified framework: we prove ball rigidity in
	the subcritical, critical, and supercritical regimes for the geometric
	classes stated below.  In the subcritical case our conclusion replaces
	the convexity assumption in the earlier symmetry theorem by strict
	$k$-convexity and strict star-shapedness; in the critical and
	supercritical cases it supplies the corresponding logarithmic and
	power-growth rigidity theorems under strict $(k-1)$-convexity and
	strict star-shapedness.

	We first introduce some definitions before presenting our results.

%We use the local version of
%\cite[Definitions~2.2--2.3]{MaZhang}.
%By \cite[Remark following Theorem~2.7]%{TrudingerWangII},
%this definition is equivalent to
%\[
%    \mu_k[u]=0
%\]
%and to the \(k\)-admissible viscosity formulation.
%If \(u\in C^{1,1}_{\mathrm{loc}}(D)\), it is also equivalent to
%\[
%    \lambda(D^2u)\in\overline{\Gamma}_k,
%    \qquad
 %   \sigma_k(D^2u)=0
 %%   \quad\text{a.e. in }D.
%\]
\begin{definition}
    A smooth bounded domain is called strictly star-shaped with respect to $x_*\in\Omega$ if 
    $$(x-x_*)\cdot\nu(x)>0\quad\forall
     x\in \partial\Omega,$$
     where $\nu$ denotes the unit outward normal to the boundary $\partial\Omega.$
     
     After translating \(x_*\) to the origin, we write
\[
h(x):=x\cdot\nu(x)
\qquad\text{on }\partial\Omega.
\]
\end{definition}
\begin{definition}
   For any open set $\Omega\subset\mathbb{R}^n,$ a function $u\in C^{2}(\Omega)$ is called k-convex if $$\lambda(D^2u(x))\in \overline{\Gamma_k}.$$  
A function \(u\in C^0(\Omega)\) is called \(k\)-convex in \(\Omega\) if
there exists a sequence of \(k\)-convex functions
\(u_m\in C^2(\Omega)\) such that
\[
u_m\longrightarrow u
\qquad\text{locally uniformly in }~\Omega
.
\]
\end{definition}
\begin{definition}
Let \(\Omega\subset\mathbb R^n\) be a bounded domain, let
$
E:=\mathbb R^n\setminus\overline\Omega,
 $
and let \(\varphi\in C^0(\partial\Omega)\).
A function \(u\in C^0(\overline E)\) is called a \(k\)-convex
solution of
\[
\begin{cases}
\sigma_k(D^2u)=0 & \text{in }E,\\
u=\varphi & \text{on }\partial\Omega,
\end{cases}
\]
if \(u|_{\partial\Omega}=\varphi\) and there exist \(k\)-convex
functions
\[
u_m\in C^2(E)\cap C^0(\overline E)
\]
such that
\[
u_m\longrightarrow u
\quad\text{in }C^0_{\mathrm{loc}}(\overline E),
\qquad
\sigma_k(D^2u_m)\longrightarrow0
\quad\text{in }L^1_{\mathrm{loc}}(E).
\]
\end{definition}
By the weak continuity of \(k\)-Hessian measures, the preceding
definition is equivalent to
\[
\mu_k[u]=0.
\]
For continuous \(k\)-convex functions it is also equivalent to
the admissible-viscosity formulation. In particular, if
\(u\in C^{1,1}_{\mathrm{loc}}(E)\), it is equivalent to
\[
\lambda(D^2u)\in\overline{\Gamma}_k,
\qquad
\sigma_k(D^2u)=0
\quad\text{a.e. in }E.
\]
See \cite[Theorem~1.1 and the remark following
Theorem~2.7]{TrudingerWangII}.

\begin{definition}
      A $C^2$ regular hypersurface $\mathcal{M}\subset\mathbb{R}^{n+1}$ is called k-convex if its principal curvature vector $\kappa(X)\in \Gamma_k$ for all $X\in\mathcal{M}.$
\end{definition}

%\begin{definition}
%   For any open set $\Omega\subset\mathbb{R}^n,$ a function $u\in C^{1,1}(\Omega)$ is called k-admissible if $$\lambda(D^2u(x))\in \overline{\Gamma_k},\quad \text{for}~a.e. x\in\Omega.$$  
%   A $C^2$ regular hypersurface $\mathcal{M}\subset\mathbb{R}^{n}$ is called k-convex if its principal curvature vector $\kappa(X)\in \Gamma_k$ for all $X\in\mathcal{M}.$
%\end{definition}

	We state the three regimes separately because their conservation laws
	and their asymptotic normalizations are genuinely different. Denote $E:=\mathbb{R}^n\setminus\overline{\Omega}.$
	
	\begin{theorem}[Subcritical rigidity]\label{thm:subcritical}
		Let $2\le k<\frac{n}{2}$ and let $\Omega\Subset\R^n$ be smooth,
		bounded, strictly star-shaped, and strictly $(k-1)$-convex.  Suppose that
		$u\in C^{1,1}_{\mathrm{loc}}(\overline E)$ is a k-admissible solution of
		\begin{equation}
		  \begin{cases}
		  		\sigma_k(D^2u)=0&\text{in }~\R^n\setminus\overline\Omega,\\
		  	u=-1& \text{on}~\partial\Omega,\\
		  	u(x)\to0\ &\text{as }|x|\to\infty.
		  \end{cases}
		\end{equation}	
	%	\[
	%	\sigma_k(D^2u)=0\ \text{in }\R^n\setminus\overline\Omega,\qquad
	%	u=-1\ \text{on }\partial\Omega,\qquad
	%	u(x)\to0\ \text{as }|x|\to\infty.
	%	\]
		If $|Du|=c>0$ on $\partial\Omega$, then for some $x_0\in\R^n$,
		\[
		\Omega=B_R(x_0),\qquad
		R=\frac{n-2k}{kc},\qquad
		u(x)=-\left(\frac{R}{|x-x_0|}\right)^{\frac{n-2k}{k}}.
		\]
	\end{theorem}
	
	\begin{theorem}[Critical rigidity]\label{thm:critical}
		Let $k=\frac{n}{2}$, and let $\Omega\Subset\R^n$ be smooth,
		bounded, strictly star-shaped, and strictly $(k-1)$-convex.  For some
		$\mathcal{A}>0$, suppose that $u\in C^{1,1}_{\mathrm{loc}}(\overline E)$  is a k-admissible solution of
			\begin{equation}
			\begin{cases}
				\sigma_k(D^2u)=0&\text{in }~\R^n\setminus\overline\Omega,\\
				u=0& \text{on}~\partial\Omega,\\
				u(x)-\mathcal{A}\log|x|=O(1)&\text{as }|x|\to\infty.
			\end{cases}
		\end{equation}	
	%	\[
	%	\sigma_k(D^2u)=0\ \text{in }E,\qquad
	%	u=0\ \text{on }\partial\Omega,\qquad
	%	u(x)-A\log|x|=O(1).
	%	\]
		If $u_\nu=|Du|=c>0$ on $\partial\Omega$, then for some
		$x_0\in\R^n$,
		\[
		\Omega=B_R(x_0),\qquad
		R=\frac{\mathcal{A}}{c},\qquad
		u(x)=\mathcal{A}\log\frac{|x-x_0|}{R}.
		\]
	\end{theorem}
	
	\begin{theorem}[Supercritical rigidity]\label{thm:supercritical}
		Let $\frac{n}{2}< k<n$, and let $\Omega\Subset\R^n$ be smooth,
		bounded, strictly star-shaped, and strictly $(k-1)$-convex.  Put
		\[
		\alpha=\frac{2k-n}{k},\qquad q=\frac{n-k}{k}=1-\alpha.
		\]
		For some $\mathfrak a>0$, suppose that the exterior  k-admissible 
		solution $u\in C^{1,1}_{\mathrm{loc}}(\overline E)$ satisfies
			\begin{equation}
			\begin{cases}
				\sigma_k(D^2u)=0&\text{in }~\R^n\setminus\overline\Omega,\\
				u=1& \text{on}~\partial\Omega,\\
				u(x)-\mathfrak a|x|^\alpha=O(1)&\text{as }|x|\to\infty.
			\end{cases}
		\end{equation}	
	%	\[
	%	\sigma_k(D^2u)=0\ \text{in }E,\qquad
	%	u=1\ \text{on }\partial\Omega,\qquad
	%	u(x)-\mathfrak a|x|^\alpha=O(1).
	%	\]
		If $u_\nu=|Du|=c>0$ on $\partial\Omega$, then for some
		$x_0\in\R^n$,
		\[
		\Omega=B_R(x_0),\qquad
		R=\left(\frac{\mathfrak a\alpha}{c}\right)^{1/q},
		\]
		and
		\[
		u(x)=\mathfrak a|x-x_0|^\alpha+1-\mathfrak aR^\alpha.
		\]
	\end{theorem}
	
	%\begin{remark}\label{rem:growth-essential}
	%	The growth condition in \cref{thm:supercritical} is part of the
	%	problem, not a consequence of the boundary data.  Even on a fixed ball,
	%	\[
	%	u_{\mathfrak a}(x)=\mathfrak a|x|^\alpha+1-\mathfrak aR^\alpha,
	%	\qquad \mathfrak a>0,
	%	\]
	%	is a one-parameter family with constant boundary gradient
	%	$\mathfrak a\alpha R^{-q}$.  The coefficient at infinity is therefore
	%	needed both for uniqueness and for the radius formula.
	%\end{remark}
    
		Besides placing the three cases in a common normalization, the proof
	has three structural features.  First, a Minkowski-gauge construction
	replaces the global distance-function barrier in the critical and
	supercritical exterior theory and thereby extends the needed analytic
	package to strictly star-shaped domains.  Second, one contact-point
	calculation yields the sharp gauge-gradient estimate in every regime.
	Third, the boundary inequality resulting from this estimate is paired
	with the conservation law appropriate to the scaling: a
	Rellich--Pohozaev identity in the subcritical case, a scale-invariant
	Wronskian--Newton flux in the critical case, and a renormalized
	Newton--Jacobi mass in the supercritical case.  Equality in the
	resulting curvature identities forces total umbilicity and hence
	sphericity.
    
	The paper is organized as follows. In  Section~\ref{sec:prelim}, we introduce
notations, the exterior analytic package, the finite-part asymptotics,
	and the compactness tools.
	In Section~\ref{sec:common-rigidity} we  prove the common ingredients
	of the rigidity argument: the sharp gauge-gradient estimate and the
	boundary curvature inequality.
	The complete proofs of
	\cref{thm:subcritical,thm:critical,thm:supercritical} are then written
	step by step, respectively, in
	\cref{sec:proof-subcritical,sec:proof-critical,sec:proof-supercritical}.
	
	\section{Preliminaries}\label{sec:prelim}
	
	\medskip
	\noindent\subsection{Newton tensors and boundary geometry.}
	
	For $\lambda=(\lambda_1,\ldots,\lambda_n)$, let $\sigma_j(\lambda)$
	be the $j$th elementary symmetric function and set $\sigma_0=1$.
	The G\aa rding cone is
	\[
	\Gamma_k=\{\lambda\in\R^n:\sigma_j(\lambda)>0,
	\ 1\le j\le k\}.
	\]
	The $j$th Newton tensor is
	\begin{equation}\label{eq:newton-def}
		T_j(A)=\sigma_j(A)\Id-\sigma_{j-1}(A)A+\cdots+(-1)^jA^j.
	\end{equation}
	We use
	\begin{align}
		\frac{\partial\sigma_{j+1}}{\partial a_{rs}}(A)&=T_j(A)^{rs},
		\label{eq:newton-derivative}\\
		T_{j-1}(A):A&=j\sigma_j(A),\label{eq:newton-euler}\\
		\tr T_j(A)&=(n-j)\sigma_j(A),\label{eq:newton-trace}\\
		T_{j-1}(A)A&=\sigma_j(A)\Id-T_j(A).\label{eq:newton-recursion}
	\end{align}
	If $A=D^2\phi$, then
	\begin{equation}\label{eq:newton-divfree}
		\partial_rT_j(D^2\phi)^{rs}=0.
	\end{equation}
%	For $A\in\Gamma_k$, $T_{k-1}(A)>0$; for
%	$A\in\overline\Gamma_k$, it is nonnegative.
%	After translating a strict star center to the origin, write
	%\begin{equation}\label{eq:support}
%		h=x\cdot\nu>0\qquad\text{on }\partial\Omega.
%	\end{equation}
Let $E=\R^n\setminus\overline\Omega$, the unit normal $\nu$ on
	$\partial\Omega$ points into $E$.  Principal curvatures are taken with
	the convention that round spheres have positive curvature, and
	\[
	H_j=\sigma_j(\kappa_1,\ldots,\kappa_{n-1}),\qquad H_0=1.
	\]
	For later use, set
	\begin{equation}\label{eq:Ij-common}
		I_j=\int_{\partial\Omega}H_j\,dS.
	\end{equation}
	The unnormalized and normalized Hsiung--Minkowski formulas
	\cite{Hsiung} are
	\begin{align}
		\int_{\partial\Omega}hH_j\,dS
		&=\frac{n-j}{j}\int_{\partial\Omega}H_{j-1}\,dS,
		\label{eq:minkowski}\\
		\int_{\partial\Omega}h\cH_j\,dS
		&=\int_{\partial\Omega}\cH_{j-1}\,dS,
		\qquad
		\cH_j=\binom{n-1}{j}^{-1}H_j.
		\label{eq:minkowski-normalized}
	\end{align}
	
	\medskip
	\noindent\subsection{The common scale function.}
	
	Throughout the proof put
	\begin{equation}\label{eq:q-def}
		q=\frac{n-k}{k}>0,
		\qquad \alpha=1-q=\frac{2k-n}{k}.
	\end{equation}
	For $t>0$ define
	\begin{equation}\label{eq:Fq}
		F_q(t)=
		\begin{cases}
			\dfrac{t^{1-q}-1}{1-q},&q\ne1,\\[5pt]
			\log t,&q=1.
		\end{cases}
	\end{equation}
	Thus $F_q(1)=0$, $F_q'(t)=t^{-q}$, and
	\begin{equation}\label{eq:Fq-hessian}
		D^2F_q(v)=v^{-q-1}
		\bigl(vD^2v-qDv\otimes Dv\bigr).
	\end{equation}
	Set
	\begin{equation}\label{eq:Mq}
		M_q[v]=vD^2v-qDv\otimes Dv.
	\end{equation}
	
	Each of \cref{thm:subcritical,thm:critical,thm:supercritical} admits a
	normalization
	\begin{equation}\label{eq:U-Fq}
		U=F_q(v),\qquad
		U=0,\quad v=1,\quad |DU|=|Dv|=b
		\quad\text{on }\partial\Omega,
	\end{equation}
	as follows.
	\begin{enumerate}[label=\textup{(\roman*)},leftmargin=2.6em]
		\item If $q>1$, set
		\begin{equation}\label{eq:normal-sub}
			U=\frac{u+1}{q-1},\qquad
			v=(-u)^{-1/(q-1)},\qquad
			b=\frac{c}{q-1}=\frac{kc}{n-2k}.
		\end{equation}
		\item If $q=1$, set
		\begin{equation}\label{eq:normal-critical}
			U=\frac{u}{\cA},\qquad v=e^U,\qquad b=\frac {c}{\cA}.
		\end{equation}
		\item If $0<q<1$, let
		\begin{equation}\label{eq:normal-super-constants}
			R_\ast=\left(\frac{\mathfrak a(1-q)}c\right)^{1/q},
			\qquad \beta=\mathfrak aR_\ast^{1-q},
			\qquad b=R_\ast^{-1},
		\end{equation}
		and set
		\begin{equation}\label{eq:normal-super}
			U=\frac{u-1}{(1-q)\beta},
			\qquad
			v=\left(\frac{u-1+\beta}{\beta}\right)^{1/(1-q)}.
		\end{equation}
		Here $(1-q)\beta b=c$.
	\end{enumerate}
	Positive affine changes preserve the homogeneous equation.  Hence in
	every case
	\begin{equation}\label{eq:unified-equation}
		M_q[v]\in\overline\Gamma_k,\qquad
		\sigma_k(M_q[v])=0\quad\text{in }E.
	\end{equation}
	For the smooth nondegenerate exterior approximations, the same
	calculation and the identity $k(q+1)=n$ give
	\begin{equation}\label{eq:approx-unified}
		\sigma_k(M_q[v_\eps])=a_\eps(x)v_\eps^n=:\mathscr R_\eps>0,
	\end{equation}
	where $a_\eps$ is a positive multiple of the original right-hand side.
	The approximation estimates used below include
	\begin{equation}\label{eq:approx-error-unified}
		\sup_E\left|v_\eps^{n+1}Dv_\eps\cdot Da_\eps\right|\longrightarrow0.
	\end{equation}
	
	\medskip

	\medskip
	\noindent\subsection{A strict $k$-convex Minkowski-gauge exhaustion.}
    
	For $0<q\le1$, Ma and Zhang's construction uses ordinary convexity to
	produce a global strict $k$-convex defining function \cite{MaZhang}.  We now replace
	that input on a strictly star-shaped domain.
	
	With a strict star center at the origin, define
	\begin{equation}\label{eq:minkowski-gauge}
		\mu(x)=\inf\{s>0:x\in s\Omega\},
		\qquad x\in\R^n\setminus\{0\}.
	\end{equation}
	Then $\mu$ is smooth away from the origin, homogeneous of degree one,
	and $\Omega=\{\mu<1\}$.
	
	\begin{lemma}[Rank-one completion]\label{lem:gauge-rank-one}
		Let $A=D^2\mu(x)$, $p=D\mu(x)$, and let
		$B=A|_{p^\perp}$.  Then
		\begin{equation}\label{eq:gauge-rank-one}
			\sigma_j(A+t\,p\otimes p)
			=\sigma_j(A)+t|p|^2\sigma_{j-1}(B),
			\qquad 1\le j\le n.
		\end{equation}
		If $\partial\Omega$ is strictly $(k-1)$-convex, then
		$\lambda(B)\in\Gamma_{k-1}$ for every $x\in E$.
	\end{lemma}
	
	\begin{proof}
		Choose an orthonormal basis $(e_1,e_2,...,e_{n-1},p/|p|)$.
        %whose last vector is $p/|p|$.
        
		Multilinearity of principal minors gives
		\[
		\frac d{dt}\sigma_j(A+tp\otimes p)
		=|p|^2\sigma_{j-1}(A|_{p^\perp}),
		\]
        then we obtain that
\[\sigma_j(A+tp\otimes p)=\sigma_j(A)+t|p|^2\sigma_{j-1}(A|_{p^\perp}).\]

		  The level set
		$\{\mu=s\}=s\partial\Omega$ has normal $p/|p|$, and for tangential
		$\tau$,
		\[
		D^2\mu(\tau,\tau)
		=|D\mu|\,\mathrm{II}_{\{\mu=s\}}(\tau,\tau).
		\]
		Strict $(k-1)$-convexity is preserved by dilation and yields the last
		assertion.
	\end{proof}
	Inspired by Xiao’s \cite{Xiao} power convexification of the Minkowski gauge, we use an exponential convexification and the rank-one completion argument.
	\begin{proposition}%[Gauge exhaustion]
    \label{prop:gauge-exhaustion}
		There exists $T_\ast>0$ such that
		\begin{equation}\label{eq:Psi}
			\Psi(x)=\frac{e^{T_\ast(\mu(x)-1)}-1}{T_\ast}
		\end{equation}
		is smooth in $E$, vanishes on $\partial\Omega$, is positive in $E$,
		and satisfies $D^2\Psi\in\Gamma_k$ throughout $E$.
	\end{proposition}
	
	\begin{proof}
		Let $x=sy$, where $s=\mu(x)\ge1$ and $y\in\partial\Omega$.
		Homogeneity gives $D\mu(sy)=D\mu(y)$ and
		$D^2\mu(sy)=s^{-1}D^2\mu(y)$.  Put
		\[
		a_j(y)=\sigma_j(D^2\mu(y)),\qquad
		b_j(y)=|D\mu(y)|^2
		\sigma_{j-1}\bigl(D^2\mu(y)|_{D\mu(y)^\perp}\bigr)>0.
		\]
		By Lemma \ref{lem:gauge-rank-one},
		\[
		\sigma_j(D^2\mu(x)+TD\mu(x)\otimes D\mu(x))
		=s^{-j}(a_j(y)+Ts\,b_j(y)).
		\]
        Since $\partial\Omega$ is compact,
		then there exists a constant $T_\ast>0$ such that $$	a_j(y)+Ts\,b_j(y)>0\quad\text{for}~y\in\partial \Omega,~1\le j\le k.
        $$
		 Since
		\[
		D^2\Psi=e^{T_\ast(\mu-1)}
		\bigl(D^2\mu+T_\ast D\mu\otimes D\mu\bigr),
		\]
		the result follows.
	\end{proof}
	
	\medskip
	\noindent\subsection{Nondegenerate subsolutions and estimates.}
	
	Choose $r_0$ and $R_0$ such that
$B_{r_0}\Subset\Omega\Subset B_{R_0/2}$. We
	use the radial model
	\begin{equation}\label{eq:critical-radial-model}
    \omega_\eps(x)=
        \begin{cases}
        -(R^2_0+\varepsilon^2)^{(q-1)/2}(|x|^2+\varepsilon^2)^{(1-q)/2},&q>1,\\
		      \frac12\log
		\frac{|x|^2+\eps^2}{R_0^2+\eps^2},&q=1,\\
    (|x|^2+\eps^2)^{(1-q)/2}
		-(R_0^2+\eps^2)^{(1-q)/2}+1,&0<q<1.
		\end{cases}
	\end{equation}
    When $q>1,$ 
    $$\sigma_j(D^2\omega_{\varepsilon})=\binom{n-1}{j-1}\frac{(q-1)^j}{j}(R_0^2+\varepsilon^2)^{j(q-1)/2}\frac{n\varepsilon^2+[n-j(q+1)]|x|^2}{(|x|^2+\varepsilon^2)^{j(q+1)/2+1}}>0,	\qquad 1\le j\le k.$$
	When $q=1$, 
	\begin{equation}\label{eq:critical-radial-sigma}
		\sigma_j(D^2 \omega_\eps)
		=\binom{n-1}{j-1}
		\frac{n\eps^2+(n-2j)|x|^2}
		{j(|x|^2+\eps^2)^{j+1}}>0,
		\qquad 1\le j\le k.
	\end{equation}
	When $0<q<1,$ 
    \begin{equation}
    \sigma_j(D^2 \omega_\varepsilon) = \binom{n-1}{j-1} \frac{ (1-q)^j \bigl( n\varepsilon^2+ (n-j(1+q))|x|^2 \bigr) }{ j(|x|^2+\varepsilon^2)^{j(1+q)/2+1} }>0,	\qquad 1\le j\le k. 
    \end{equation}
	In both cases
	\begin{equation}\label{eq:radial-rhs}
		\sigma_k(D^2 \omega_\eps)=F_{\varepsilon}
		=c_{n,k,q}\eps^2(|x|^2+\eps^2)^{-n/2-1}>0.
	\end{equation}
    In the following Proposition, we will construct the global strict subsolution and then obtain the existence and properties of exterior solutions when $2\leq k<n.$
The role of a strict admissible subsolution in obtaining boundary
second-derivative estimates for fully nonlinear Dirichlet problems
goes back to the subsolution method of Guan \cite{Guan94}.
Here our main task is to construct such a global strict subsolution
under the weaker assumption of strict star-shapedness.

%	On a fixed annulus, a sufficiently small multiple of $\Psi$ lies above
%	$w_\eps$ near $\partial\Omega$ and below it near the outer edge.
%	The regularized maximum construction, together with concavity and
%	monotonicity of $\sigma_k^{1/k}$, therefore gives a smooth strict
%	$k$-subsolution which agrees with $\eta\Psi$ (respectively
%	$1+\eta\Psi$) near the inner boundary and with $w_\eps$ outside a
%	fixed ball.  Guan's bounded-domain theorem \cite{Guan94} then solves
%	the nondegenerate problem on truncated exterior rings.
	
%	For later reference we record the resulting package. 
	
	\begin{proposition}
		\label{prop:star-package}
		%Assume $0<q\le1$, and
        Let $\Omega\subset\R^n$ be smooth,
		bounded, strictly star-shaped, and strictly $(k-1)$-convex.
Then the normalized exterior problem has a unique \(k\)-admissible solution \[ u\in C^{1,1}_{\mathrm{loc}}(\overline E), \qquad E=\mathbb R^n\setminus\overline\Omega, \]   
in each of the following classes: 
\begin{equation}\label{eq11prop2.3}
 \text{When}~q>1,
			\begin{cases}
				\sigma_k(D^2u)=0&\text{in }~\R^n\setminus\overline\Omega,\\
				u=-1& \text{on}~\partial\Omega,\\
				u(x)\to 0&\text{as}~|x|\to\infty.
			\end{cases}
		\end{equation}	
\begin{equation}\label{eq1prop2.3}
\text{When}~q=1,
			\begin{cases}
				\sigma_k(D^2u)=0&\text{in }~\R^n\setminus\overline\Omega,\\
				u=0& \text{on}~\partial\Omega,\\
				u(x)-\log|x|=O(1)&\text{as }|x|\to\infty.
			\end{cases}
		\end{equation}	
        \begin{equation}\label{eq2prop2.3}
    \text{When}~0<q<1,
			\begin{cases}
				\sigma_k(D^2u)=0&\text{in }~\R^n\setminus\overline\Omega,\\
				u=1& \text{on}~\partial\Omega,\\
				u(x)-|x|^\alpha=O(1)&\text{as }|x|\to\infty.
			\end{cases}
		\end{equation}	
        
        The solution is the local \(C^1\)-limit on \(\overline E\) of smooth strictly admissible solutions $u_{\varepsilon
        }$ of the nondegenerate equations 
        \begin{equation}
        \text{When}~q>1,
			\begin{cases}
				\sigma_k(D^2u_{\varepsilon})=F_{\varepsilon}&\text{in }~\R^n\setminus\overline\Omega,\\
				u_{\varepsilon}=-1& \text{on}~\partial\Omega,\\
				u_{\varepsilon}(x)\to 0&\text{as}~|x|\to\infty.
			\end{cases}
		\end{equation}	

\begin{equation}\label{eq1'prop2.3}
 \text{When}~q=1,			\begin{cases}
				\sigma_k(D^2u_{\varepsilon})=F_{\varepsilon}&\text{in }~\R^n\setminus\overline\Omega,\\
				u_{\varepsilon}=0& \text{on}~\partial\Omega,\\
				u_{\varepsilon}(x)-\log|x|=O(1)&\text{as }|x|\to\infty
			\end{cases}
		\end{equation}	
        \begin{equation}\label{eq2'prop2.3}
\text{When}~0<q<1,		\begin{cases}
				\sigma_k(D^2u_{\varepsilon})=F_{\varepsilon}&\text{in }~\R^n\setminus\overline\Omega,\\
				u_{\varepsilon}=1& \text{on}~\partial\Omega,\\
				u_{\varepsilon}(x)-|x|^\alpha=O(1)&\text{as }|x|\to\infty
			\end{cases}
		\end{equation}	
with right-hand side \[ F_\varepsilon(x) =c_{n,k,q}\varepsilon^2 \bigl(|x|^2+\varepsilon^2\bigr)^{-n/2-1}. \] 
There exists \(C\ge1\), independent of
\(0<\varepsilon\le\varepsilon_0\) and of the truncation radius \(R\),
such that
\begin{equation}\label{eq:package-estimates-revised}
\begin{gathered}
\begin{cases}
 C^{-1}|x|^{1-q}\le -u_\varepsilon(x)\le C|x|^{1-q},
     &q>1,\\
 |u_\varepsilon(x)-\log |x||\le C,
     &q=1,\\
 |u_\varepsilon(x)-|x|^{1-q}|\le C,
     &0<q<1,
\end{cases}\\
 C^{-1}|x|^{-q}\le |Du_\varepsilon(x)|\le C|x|^{-q},
 \qquad
 |D^2u_\varepsilon(x)|\le C|x|^{-q-1}.
\end{gathered}
\end{equation}
%The approximating solutions satisfy, uniformly in the approximation parameters, \begin{equation}\label{eq:package-estimates-revised}
%|Du_{\varepsilon}(x)|\asymp |x|^{-q}, \qquad |D^2u_{\varepsilon}(x)|\leq C|x|^{-q-1}. 
%\end{equation}
Finally, all integrations by parts below are valid for the approximants and pass to the limit by the uniform estimates and asymptotics.
After translating a strict star center to the origin, there exists a positive constant $c_0$ independent of $\varepsilon$ such that
\begin{equation}\label{eq:radial-nondegeneracy-revised} x\cdot Du_{\varepsilon}(x)\geq \begin{cases} c_0,&q=1,\\ c_0|x|^{1-q},&q\ne 1. \end{cases} \end{equation} 
The corresponding transformed approximants satisfy \begin{equation}\label{eq:approx-error-unified-revised} \sup_E \left| v_\varepsilon^{n+1} Dv_\varepsilon\cdot Da_\varepsilon \right| \longrightarrow0. \end{equation} 
If \(|Du|=c>0\) on \(\partial\Omega\), then
\[
\|\,|Du_\varepsilon|-c\,\|_{L^\infty(\partial\Omega)}
\longrightarrow0.
\]
In particular,
\[
\max_{\partial\Omega}|Du_\varepsilon|\longrightarrow c.
\]
The corresponding statements with an arbitrary positive logarithmic or power-growth coefficient follow by a positive affine rescaling. 
	\end{proposition}
	
	\begin{proof}
Firstly, we construct  the global strict subsolution for equations \eqref{eq11prop2.3}, \eqref{eq1prop2.3} and \eqref{eq2prop2.3}.
Let \[ \ell= \begin{cases} -1,&q>1,\\0,&q=1,\\ 1,&0<q<1. \end{cases} \]
By the definition of $w_{\varepsilon}$, 
there exist a fixed collar \(\mathcal N\) of \(\partial\Omega\), a fixed outer annulus $ \mathcal A_{\mathrm{out}} =B_{3R_0}\setminus\overline{B_{2R_0}}, $ and constants \(d_0,d_1>0\), independent of all sufficiently small \(\varepsilon\), such that
\begin{equation}\label{wN}
     \omega_\varepsilon\leq\ell-d_0 \quad\text{in }\mathcal N,
\end{equation}
and
\begin{equation}\label{wA}
      \omega_\varepsilon\geq\ell+d_1 \quad\text{in }\mathcal A_{\mathrm{out}}.
\end{equation}
Denote \[ K=\overline{B_{3R_0}\setminus\Omega}. \]
Choose \(\eta=\eta(R_0,q,n)>0\), independently of \(\varepsilon\), so small that 
\begin{equation}\label{etaPsi}
    \eta\max_K\Psi\leq\frac{d_1}{4}.
\end{equation}
Define \[ g=\ell+\eta\Psi.\]
Then by \eqref{wN},
\[ g- \omega_\varepsilon=\ell- \omega_{\varepsilon}+\eta\Psi\geq d_0 \quad\text{near }\partial\Omega, \] 
while by \eqref{wA} and \eqref{etaPsi},
\[ \omega_\varepsilon-g=\omega_{\varepsilon}-\ell-\eta\Psi\geq d_1-\eta\Psi\geq\frac{3d_1}{4} \quad\text{in }\mathcal A_{\mathrm{out}}. \]
By Proposition \ref{prop:gauge-exhaustion},
\[ m_*:=\min_K\sigma_k(D^2\Psi)>0. \]
Then for $\varepsilon$ sufficiently small, there is 
\[ \sigma_k(D^2g) =\eta^k\sigma_k(D^2\Psi)\geq \eta^k m_*\geq C\varepsilon^2\geq F_{\varepsilon}\quad\text{in}~K.\] 
Choose a fixed smoothing width 
\[ 0<\delta< \frac14\min\{d_0,d_1\}. \] Apply Guan's regularized-maximum construction \cite[Lemma~3.2]{Guan PF} to \(g\) and \(w_\varepsilon\), there exists a smooth function $\mathcal M_\delta(g,\omega_\varepsilon)\geq \max{(g,\omega_{\varepsilon})}$, 
%Denote the resulting smooth maximum by \[ \mathcal M_\delta(g,w_\varepsilon). \] 
$\mathcal M_\delta(g,w_\varepsilon)$ agrees with \(g\) near \(\partial\Omega\), 
agrees with \(w_\varepsilon\) in a neighborhood of \(\mathcal A_{\mathrm{out}}\), and satisfies
\[ D^2\mathcal M_\delta \geq \theta D^2g+(1-\theta)D^2\omega_\varepsilon \] 
for some \(0\leq\theta\leq1\). 
Because \(\Gamma_k\) is convex and \(\sigma_k^{1/k}\) is concave in \(\Gamma_k\),
\[ \begin{aligned}
\sigma_k^{1/k}(D^2\mathcal M_\delta) &\geq \sigma_k^{1/k} \bigl( \theta D^2g+ (1-\theta)D^2\omega_\varepsilon \bigr)\\ &
\geq \theta\sigma_k^{1/k}(D^2g) +(1-\theta) \sigma_k^{1/k}(D^2\omega_\varepsilon)\\ &\geq F_\varepsilon^{1/k}. \end{aligned} \] Thus
\[ \sigma_k(D^2\mathcal M_\delta) \geq F_\varepsilon. \] 
We may now define \[ \underline u_\varepsilon(x)= \begin{cases} \mathcal M_\delta(g,\omega_\varepsilon)(x), &x\in B_{3R_0}\setminus\overline\Omega,\\ \omega_\varepsilon(x), &x\in\mathbb R^n\setminus B_{3R_0}. \end{cases} \] 
The two definitions agree on an open annulus, so \(\underline u_\varepsilon\) is smooth. 
It satisfies 
\[ D^2\underline u_\varepsilon\in\Gamma_k,\qquad \sigma_k(D^2\underline u_\varepsilon) \geq F_\varepsilon \quad\text{in }E, \] 
\[ \underline u_\varepsilon=\ell \quad\text{on }\partial\Omega, \] 
and
\[ \underline u_\varepsilon =\ell+\eta\Psi \quad\text{in a fixed collar of }\partial\Omega, \] 
\[ \underline u_\varepsilon=\omega_\varepsilon \quad\text{outside }B_{3R_0}. \] %Moreover, \[ \underline u_\varepsilon\geq\ell \quad\text{throughout }E. \] 
The collar, the smoothing width, and the \(C^m\)-bounds of \(\underline u_\varepsilon\) on the gluing region are independent of \(\varepsilon\). Thus we obtain the global strict subsolution $\underline{u}_{\varepsilon}.$ 

 The proof after
	the subsolution construction is the estimate scheme of
	\cite{MaZhang}, the details below identify the geometric inputs.
For $R>4R_0$, solve on $B_R\setminus\overline\Omega$ the strictly elliptic
Dirichlet problem with inner value $\ell$ and outer value
$\omega_\varepsilon|_{\partial B_R}$.  The preceding strict subsolution
and the bounded-domain subsolution theorem of Guan
\cite[Theorem~1.1]{Guan94} give a smooth strictly admissible solution $u_{\varepsilon,R}$.
We now check that the estimates used to let $R\to\infty$ and then
$\varepsilon\downarrow0$ remain valid under the present geometry.
The $C^0$ comparison and the upper gradient estimate use only the radial
outer barriers and the strict subsolution.  The lower estimate for
$x\cdot Du_{\varepsilon,R}$ uses the linearized auxiliary functions of
Ma--Zhang; its inner-boundary sign requires precisely
$x\cdot\nu\ge h_0>0$, while its outer-boundary sign is radial.  Tangential
second derivatives at the inner boundary use strict
$(k-1)$-convexity, and the mixed and double-normal estimates use the
fixed strict subsolution.  The interior and outer-boundary estimates do
not use the inner-domain convexity.  Thus the arguments in
\cite[Sections~3--5]{MaZhang} apply verbatim after replacing their
convex defining function by $\underline u_\varepsilon$.  They yield,
uniformly in $R$ and $\varepsilon$,
\[
 |Du_{\varepsilon,R}(x)|\asymp |x|^{-q},
 \qquad
 |D^2u_{\varepsilon,R}(x)|\le C|x|^{-q-1},
\]
and
\[
 x\cdot Du_{\varepsilon,R}(x)\ge
 \begin{cases}c_0,&q=1,\\c_0|x|^{1-q},&0<q<1.\end{cases}
\]
Passing first to $R=\infty$ and then to $\varepsilon=0$ gives the stated
solution, compactness, and estimates.  Comparison on large truncated
rings, followed by $R\to\infty$, gives uniqueness in the prescribed
growth class.  It also gives $u\ge0$ in the critical normalization and
$u\ge1$ in the supercritical normalization.

It remains only to justify the error estimate used later.  Under the
affine normalizations, $a_\varepsilon$ is a fixed positive multiple of
$F_\varepsilon$, and, since the exterior domain stays a positive distance
from the origin,
\[
 |Da_\varepsilon(x)|\le C\varepsilon^2|x|^{-n-3}.
\]
The preceding estimates and the definition of the transform give
$v_\varepsilon\asymp|x|$ and $|Dv_\varepsilon|\le C$.  Hence
\[
 v_\varepsilon^{n+1}|Dv_\varepsilon|\,|Da_\varepsilon|
 \le C\varepsilon^2|x|^{-2}\le C\varepsilon^2,
\]
which proves \eqref{eq:approx-error-unified-revised}.

	\end{proof}
	
	The radial derivative in \eqref{eq:radial-nondegeneracy-revised} shows that the
	critical solution is nonnegative and the supercritical solution is at
	least one in $E$.  Thus all powers used in \eqref{eq:normal-super} are
	well defined.
	
	\medskip
	\noindent\subsection{The finite part at infinity.}
	
	The $O(1)$ remainder in the critical and supercritical growth classes
	has a genuine limit.  This fact is needed to identify the flux at
	infinity.
	
	\begin{proposition}[Finite-part asymptotics]\label{prop:finite-part}
		Assume $0<q\le1$ and use the normalized function $U$ from
		\eqref{eq:normal-critical} or \eqref{eq:normal-super}.  There is
		$\gamma\in\R$ such that
		\begin{equation}\label{eq:finite-part}
			U(x)-F_q(b|x|)\longrightarrow\gamma
			\qquad\text{uniformly as }|x|\to\infty.
		\end{equation}
		Moreover,
		\begin{equation}\label{eq:finite-part-gradient}
			DU(x)-D\bigl(F_q(b|x|)\bigr)=o(|x|^{-q}).
		\end{equation}
		Consequently,
		\begin{equation}\label{eq:v-asymptotics}
			\begin{cases}
				v(x)=e^\gamma b|x|(1+o(1)),\quad
				Dv(x)=e^\gamma b\dfrac{x}{|x|}+o(1),&q=1,\\[6pt]
				v(x)=b|x|+o(|x|),\quad
				Dv(x)=b\dfrac{x}{|x|}+o(1),&0<q<1.
			\end{cases}
		\end{equation}
	\end{proposition}
	
	\begin{proof}
		Put $\Phi(x)=F_q(b|x|)$ and $f=U-\Phi$.  The exterior estimates give
		\begin{equation}\label{eq:f-C11}
			|f|\le C,\qquad |D^2f(x)|\le C|x|^{-q-1}.
		\end{equation}
		Since $U$ and $\Phi$ solve the homogeneous equation,
		\begin{equation}\label{eq:linearized-f}
			a^{ij}f_{ij}=0,
			\qquad
			a^{ij}=\int_0^1T_{k-1}^{ij}
			\bigl((1-t)D^2\Phi+tD^2U\bigr)\,dt.
		\end{equation}
		The Hessian of $\Phi$ has $n-1$ tangential eigenvalues
		$a_r=b^{1-q}r^{-q-1}$ and one radial eigenvalue $-qa_r$.  The radial
		eigenvalue of $T_{k-1}(D^2\Phi)$ is
		$\binom{n-1}{k-1}a_r^{k-1}$, while every tangential eigenvalue is
		\[
		\frac q{n-1}\binom{n-1}{k-1}a_r^{k-1}>0.
		\]
		Convexity of $\overline\Gamma_k$, positivity on a fixed initial part
		of the segment in \eqref{eq:linearized-f}, and
		\eqref{eq:package-estimates-revised} show that, on every doubled annulus,
		\begin{equation}\label{eq:linear-uniform}
			\lambda r^{-(q+1)(k-1)}\Id
			\le(a^{ij})\le
			\Lambda r^{-(q+1)(k-1)}\Id,
		\end{equation}
		with constants independent of the radius.
		
		Let $M(R)=\max_{|x|=R}f$ and $m(R)=\min_{|x|=R}f$.  Comparing $U$ with
		the exact solutions obtained by multiplying the fundamental term in
		$\Phi$ by $1\pm\delta$ and adjusting the constant at $|x|=R$, then
		letting the outer radius tend to infinity and $\delta\downarrow0$,
		gives
		\begin{equation}\label{eq:monotone-remainders}
			f(x)\le M(R),\qquad f(x)\ge m(R)
			\quad (|x|\ge R).
		\end{equation}
		Thus $M$ decreases and $m$ increases.  Apply the scale-invariant
		Harnack inequality to the positive function
		$f-m(R/2)+\eta$ on $R/2<|x|<2R$:
		\[
		M(R)-m(R/2)+\eta
		\le C_H\bigl(m(R)-m(R/2)+\eta\bigr).
		\]
		Letting $\eta\downarrow0$ and $R\to\infty$ shows that the two limiting
		values coincide.  This proves \eqref{eq:finite-part}.
		
		If
		$\omega(R)=\sup_{R/2<|x|<2R}|f-\gamma|$, then Taylor's formula and
		\eqref{eq:f-C11} give, for $|x|\simeq R$,
		\[
		|Df(x)|\le C\left(\frac{\omega(R)}\rho+\rho R^{-q-1}\right).
		\]
		For $q=1$ choose $\rho=R\sqrt{\omega(R)}$; for $q<1$ choose
		$\rho=\sqrt{\omega(R)R^{q+1}}$, truncating it at $R/10$ if necessary.
		This yields \eqref{eq:finite-part-gradient}.  The formulas in
		\eqref{eq:v-asymptotics} follow by inverting $F_q$.
	\end{proof}
	
	\medskip
	\noindent\subsection{Weak continuity of Newton tensors.}
	
	The critical and supercritical fluxes must be passed from smooth
	nondegenerate approximants to a $C^{1,1}$ limit.  We use the following
	single compactness statement in both regimes.
	
	\begin{lemma}[Weak continuity of Newton tensors]\label{lem:weak-newton}
		Suppose $U_m,U\in C^{1,1}_{\mathrm{loc}}(D)$ are locally uniformly
		bounded in $C^{1,1}$ and $U_m\to U$ locally uniformly.  Then
		$DU_m\to DU$ locally uniformly and, for every $j$,
		\begin{equation}\label{eq:weak-newton}
			T_j(D^2U_m)\stackrel{\ast}{\rightharpoonup}T_j(D^2U)
			\quad\text{in }L^\infty_{\mathrm{loc}}.
		\end{equation}
	\end{lemma}
	
	\begin{proof}
		The gradient convergence follows from the elementary interpolation
		estimate between $C^0$ and $C^{1,1}$.  Entries of $T_j(D^2U)$ are
		linear combinations of Hessian minors.  With the generalized Kronecker
		symbol, integration by parts writes each such minor as one first
		derivative of a test function times $DU$ and a minor of one lower
		order.  Antisymmetry cancels the terms in which the derivative falls
		on another Hessian factor.  Induction on $j$ gives distributional
		convergence, and the uniform Hessian bounds upgrade it to weak-star
		convergence.  This is the Hessian null-Lagrangian property; see
		\cite{TrudingerWangII,TrudingerWangIII}.
	\end{proof}
	
	\section{Gradient estimates and boundary curvature inequality}
	\label{sec:common-rigidity}
	
	\medskip
	\noindent\subsection{The sharp gradient estimate.}
	
	The central estimate in all three regimes is the following.
	
	\begin{lemma}[Gauge-gradient maximum principle]\label{thm:gradient}
		Under the hypotheses of any one of
		\cref{thm:subcritical,thm:critical,thm:supercritical}, the normalized gauge in
		\eqref{eq:U-Fq} satisfies
		\begin{equation}\label{eq:gradient-bound}
			|Dv|\le b\qquad\text{in }E.
		\end{equation}
	\end{lemma}
	
	\begin{proof}
		We work first with a smooth strictly admissible approximation satisfying
		\eqref{eq:approx-unified}.  Suppress the approximation index, write
		\[
		T^{ij}=T_{k-1}^{ij}(M_q[v]),
		\]
		and define
		\begin{equation}\label{eq:linearized-Mq}
			\cL\phi=vT^{ij}\phi_{ij}-2qT^{ij}v_i\phi_j.
		\end{equation}
		Differentiating \eqref{eq:approx-unified} and using Euler's identity
		gives
		\begin{align}
			\cL v_\ell
			&=\left(na v^{n-1}-T^{ij}v_{ij}\right)v_\ell
			+v^n a_\ell,\label{eq:L-vl}\\
			vT^{ij}v_{ij}-qT^{ij}v_iv_j&=k\mathscr R.
			\label{eq:euler-Mq}
		\end{align}
		
		Let
		\begin{equation}\label{eq:theta}
			\vartheta=
			\begin{cases}
				1,&q\ge1,\\
				-1,&0<q<1.
			\end{cases}
		\end{equation}
		 Motivated by the P-function constructed in \cite[Theorem~4.7]{MaZhang}, we consider
		\begin{equation}\label{eq:Pdelta}
			P_\delta=\frac12|Dv_{\varepsilon}|^2-\vartheta\delta\log v_{\varepsilon},
			\qquad \delta>0.
		\end{equation}
        Let $L>0$ and $\Omega\subset B_L.$ There exists $\delta_L$ sufficiently small, such that 
        \begin{equation}\label{deltaL}
            |\delta\log v_{\varepsilon}|<\frac{1}{8}b^2\quad\text{in}~\overline{B_L\setminus\Omega}~\text{for}~0<\delta<\delta_L.
        \end{equation}
        Fix $0<\delta<\min\{\delta_L,\frac{1}{4}qb^2\},$
        suppose that there exists $x_*\in B_L\setminus\overline{\Omega} $ such that
        $P_{\delta}(x_*)=\max_{\overline{B_L\setminus \Omega}}P_{\delta}.$
       
        For $\varepsilon$ sufficiently small, there is 
        $$\frac{1}{2}s^2-\vartheta\delta\log v_{\varepsilon}(x_*)>\frac{1}{2}b^2.$$
 According to \eqref{deltaL}, we have
 \begin{equation}\nonumber
     \frac{1}{2}s^2>\frac{1}{4}b^2,
 \end{equation}
this implies that 
\begin{equation}\label{est of s}
    s>\frac{b}{\sqrt{2}}>0.
\end{equation}
Then
        $Dv_{\varepsilon}(x_*)\ne 0.$
		We choose coordinates so
		that $Dv_{\varepsilon}(x_*)=se_n$ and diagonalize the tangential block of $D^2v_{\varepsilon}(x_*)$.  The
		condition $DP_\delta(x_*)=0$ gives
		\[
		D^2v_{\varepsilon}(x_*)Dv_{\varepsilon}(x_*)=\frac{\vartheta\delta}{v}Dv_{\varepsilon}(x_*),
		\]
		and hence
		\begin{equation}\label{eq:contact-Hessian}
			v_{\varepsilon}(x_*)D^2v_{\varepsilon}(x_*)=\diag(\lambda_1,\ldots,\lambda_{n-1},
			\vartheta\delta).
		\end{equation}
		Let
		\begin{equation}\label{eq:r-contact}
			r=qs^2-\vartheta\delta,
			\qquad
			S_j=\sigma_j(\lambda_1,\ldots,\lambda_{n-1}).
		\end{equation}
        By $\delta<\frac{1}{4}qb^2$ and \eqref{est of s}, 
        we have 
        \begin{equation}\label{est of r}
            r>\frac{1}{4}qb^2>0
        \end{equation}
        uniformly in sufficiently small $\varepsilon.$
		 Since
		\[
		M_q[v_{\varepsilon}(x_*)]=v_{\varepsilon}(x_*)D^2v_{\varepsilon}(x_*)-qDv_{\varepsilon}\otimes Dv_{\varepsilon}=\diag(\lambda_1,\ldots,\lambda_{n-1},-r)
		\in\Gamma_k
		\]
		and
		\[
		S_j-rS_{j-1}=\sigma_j(M_q[v_{\varepsilon}])>0\quad\text{for}~1\leq j\leq k,
		\]
		induction shows that
		\begin{equation}\label{eq:tangential-admissible}
			(\lambda_1,\ldots,\lambda_{n-1})\in\Gamma_k.
		\end{equation}
		The equation becomes
		\begin{equation}\label{eq:Sk-relation}
			S_k=rS_{k-1}+\mathscr R.
		\end{equation}
		
		We use the elementary identities
		\begin{align}
			\sum_{\alpha=1}^{n-1}\lambda_\alpha^2
			\sigma_{k-1}(\lambda|\alpha)
			&=S_1S_k-(k+1)S_{k+1},\label{eq:quadratic-newton-1}\\
			\sum_{\alpha=1}^{n-1}\lambda_\alpha^2
			\sigma_{k-2}(\lambda|\alpha)
			&=S_1S_{k-1}-kS_k.\label{eq:quadratic-newton-2}
		\end{align}
		They follow by separating repeated-index and distinct-index monomials
		in $S_1S_k$ and $S_1S_{k-1}$.
		
		A direct calculation using \eqref{eq:L-vl}--\eqref{eq:Sk-relation}
		and \eqref{eq:quadratic-newton-1}--\eqref{eq:quadratic-newton-2}
		gives the exact contact identity
		\begin{align}
			(v_{\varepsilon}\cL P_\delta)(x_*)
			={}&\left[
			\left(k-\frac1q\right)r^2
			+\frac{\vartheta(q-1)}q r\delta+2\delta^2
			\right]S_{k-1}-(k+1)S_{k+1}\notag\\
			&+\mathscr R(S_1+2kr)
			+v_{\varepsilon}^{n+1}(x_*)Dv_{\varepsilon}(x_*)\cdot Da.
			\label{eq:exact-contact1}
		\end{align}
		For completeness, the terms involving $Da$ arise only from
		differentiating $a(x)v^n$, while the two contributions linear in
		$\mathscr R r$ combine to $2kr\mathscr R$.
		
		Newton--Maclaurin in the $n-1$ tangential variables gives
		\begin{equation}\label{eq:newton-Skp1}
			(k+1)S_{k+1}
			\le \frac{k(n-k-1)}{n-k}\frac{S_k^2}{S_{k-1}}
			=\left(k-\frac1q\right)\frac{S_k^2}{S_{k-1}}.
		\end{equation}
		Using \eqref{eq:Sk-relation} in \eqref{eq:exact-contact1}, we obtain
		\begin{align}
			(v\cL P_\delta)(x_*)
			\ge{}&\left(
			\frac{|q-1|}{q}r\delta+2\delta^2
			\right)S_{k-1}\notag\\
			&+\mathscr R\left[
			S_1+\frac{2r}{q}
			-\left(k-\frac1q\right)\frac{\mathscr R}{S_{k-1}}
			\right]
			+v_{\varepsilon}^{n+1}(x_*)Dv_{\varepsilon}(x_*)\cdot Da.
			\label{eq:contact-lower}
		\end{align}
		The quotient form of Newton--Maclaurin yields
		\begin{equation}\label{eq:S1-lower}
			S_1\ge\frac{n-1}{q}\frac{S_k}{S_{k-1}}.
		\end{equation}
        By \eqref{eq:Sk-relation} and \eqref{eq:S1-lower}, we obtain that
\begin{equation}\label{eq:R-positive}
		S_1+\frac{2r}{q}
			-\left(k-\frac1q\right)\frac{\mathscr R}{S_{k-1}}\geq 	\frac{n+1}{q}r+\frac{k}{q}
			\frac{\mathscr R}{S_{k-1}}>0.
		\end{equation}
 From $S_j>rS_{j-1}$ for $1\leq j\leq k$ and \eqref{est of r},
		we have \begin{equation}\label{eq:Skm1-lower}
			S_{k-1}\ge r^{k-1}>(\frac{1}{4}qb^2)^{k-1}>0,
		\end{equation}
Thus we have
$$(\frac{|q-1|}{q}r\delta+2\delta^2)S_{k-1}>2\delta^2(\frac{1}{4}qb^2)^{k-1}> 0.$$

		The first term in \eqref{eq:contact-lower} is therefore bounded below
		by a positive constant depending on $\delta$, whereas the last term
		tends uniformly to zero by \eqref{eq:approx-error-unified-revised}.  Then for sufficiently small $\varepsilon$ we have
 $\cL P_\delta(x_*)>0$ .
		This contradicts with $DP_\delta(x_*)=0$, $D^2P_\delta(x_*)\le0$, and $T>0$.
        
		Then we prove that $P_{\delta}$ can not achieve its maximum on $\partial B_L$ when $L$ is sufficiently large.
          When $q>1$, by \eqref{eq:package-estimates-revised},
\(v_\varepsilon\asymp |x|\) and \(|Dv_\varepsilon|\le C\).
Consequently,
\[
P_\delta(x)\le C-\delta\log(c|x|)\longrightarrow-\infty
\qquad\text{as }|x|\to\infty .
\] 
    %    When $q>1$,
	%\eqref{eq:package-estimates-revised} implies $v(x)=a_0|x|+o_2(|x|)$;
    When $q=1$,
		use \eqref{eq:v-asymptotics}.  Thus $|Dv_{\varepsilon}|$ is bounded and
		$P_\delta(x)\to-\infty$ as $|x|\to\infty$.  
        Thus when $q\geq 1,$ 
        \begin{equation}
            P_{\delta}\leq \max_{\partial\Omega}P_{\delta}=\frac{1}{2}\max_{\partial\Omega}|Dv_{\varepsilon}|^2\quad\text{in}~\overline{B_L\setminus\Omega}.
        \end{equation}
      Successively letting $\varepsilon$ and $\delta$ converge to 
0, and letting $L$
converge to infinity, we obtain 
	\eqref{eq:gradient-bound}.
		
		When $0<q<1$, $P_\delta=|Dv_{\varepsilon}|^2/2+\delta\log v_{\varepsilon}$ does not decay at
		infinity.  Let
		$\delta_L'=(\log L)^{-2}$. Since $v_{\varepsilon}=1$ on $\partial\Omega,$
$$P_{\delta}=\frac{1}{2}|Dv_{\varepsilon}|^2\quad\text{on}~\partial\Omega.$$
For $0<\delta<\delta_L',$
 according to \eqref{eq:v-asymptotics}, there is 
        \[
		P_{\delta}=\frac{b^2}{2}+o(1)
		\quad\text{on }\partial B_L.
		\]
	 Successively letting $\varepsilon$ and $\delta$ converge to 
0, and letting $L$
converge to infinity. This proves \eqref{eq:gradient-bound} in the remaining
		case.
	\end{proof}
	
	\medskip
	%\noindent\textit{Remark.}
	%The term
	%$\vartheta(q-1)r\delta/q=|q-1|r\delta/q$ explains the choice of sign in
	%\eqref{eq:Pdelta}.  At $q=1$ this first-order term vanishes, but the
	%strictly positive term $2\delta^2S_{k-1}$ remains.  Thus the critical
	%case is not lost in the limiting contact algebra.
	
	\medskip
	\noindent\subsection{Boundary curvature inequality.}
	
	The limiting exterior solution is only $C^{1,1}$, so a classical
	pointwise value of $v_{\nu\nu}$ on $\partial\Omega$ is not available.
	The next argument avoids such a trace.
	
	\begin{proposition}[Boundary curvature inequality]\label{prop:boundary-curvature}
		Under the hypotheses of any one of
		\cref{thm:subcritical,thm:critical,thm:supercritical},
		\begin{equation}\label{eq:boundary-curvature}
			H_k\ge qbH_{k-1}\qquad\text{on }\partial\Omega.
		\end{equation}
		In all three regimes, the initially strictly
		$(k-1)$-convex boundary is therefore automatically strictly
		$k$-convex.
	\end{proposition}
	
	\begin{proof}
		Let $d(x)=\dist(x,\Omega)$ in a sufficiently small exterior tubular
		neighborhood.  If $x=y+t\nu(y)$, \cref{thm:gradient} gives
		\begin{equation}\label{eq:distance-upper}
			v(x)-1
			=\int_0^tDv(y+s\nu(y))\cdot\nu(y)\,ds
			\le bt,
		\end{equation}
		so $v\le1+bd$.
		
		Fix $x_0\in\partial\Omega$ and $\eta>0$.  Put
		\[
		\Phi_\eta(x)=1+bd(x)+\eta|x-x_0|^2.
		\]
		For a fixed collar width $\rho_0$ and $\tau>0$, define
		\begin{equation}\label{eq:perturbed-contact}
			\Phi_{\eta,\tau}
			=\Phi_\eta-\tau d+\frac{2\tau}{\rho_0}d^2.
		\end{equation}
		On $d=0$ and $d=\rho_0$, the difference
		$v-\Phi_{\eta,\tau}$ is nonpositive.  Since $v$ is $C^{1,1}$ and has
		the same first jet as $1+bd$ at $x_0$, along the normal ray
		\[
		v(x_0+t\nu)-\Phi_\eta(x_0+t\nu)\ge-C_\eta t^2.
		\]
		Taking $t=\tau/(2C_\eta)$ shows that
		$v-\Phi_{\eta,\tau}$ is positive somewhere in the collar.  Hence,
		after adding its positive maximum to $\Phi_{\eta,\tau}$, we obtain a
		smooth upper test for $v$ at an interior point $x_{\eta,\tau}$.  From
		\eqref{eq:distance-upper},
		\[
		|x_{\eta,\tau}-x_0|^2\le\frac{\tau\rho_0}{\eta},
		\]
		so the contact points converge to $x_0$ as $\tau\downarrow0$.
		
		The admissible viscosity subsolution condition for
		\eqref{eq:unified-equation} gives at the interior contacts
		\[
		\sigma_k\bigl(\phi D^2\phi-qD\phi\otimes D\phi\bigr)\ge0.
		\]
		Letting first $\tau\downarrow0$ and then $\eta\downarrow0$ yields
		\begin{equation}\label{eq:boundary-viscosity}
			\sigma_k\bigl(bD^2d(x_0)-qb^2\nu\otimes\nu\bigr)\ge0.
		\end{equation}
		In principal coordinates,
		\[
		Dd=\nu,
		\qquad
		D^2d=\diag(\kappa_1,\ldots,\kappa_{n-1},0).
		\]
		Thus \eqref{eq:boundary-viscosity} is
		\[
		0\le b^k(H_k-qbH_{k-1}),
		\]
		which proves \eqref{eq:boundary-curvature}.
	\end{proof}
	
	\medskip
	\noindent\textbf{The common boundary-to-sphere closure.}
	
	All three regime-specific arguments will produce the same integral
	ratio.  The next proposition combines it with
	\eqref{eq:boundary-curvature} and is the only geometric rigidity step
	needed in the three final proofs.
	
	\begin{proposition}[Boundary quotient rigidity]\label{prop:geometric-closure}
		Let $2\le k<n$.  Suppose $\partial\Omega$ is smooth, connected,
		strictly star-shaped, and strictly $k$-convex.  If, for some $b>0$,
		\begin{equation}\label{eq:closure-pointwise}
			H_k\ge qbH_{k-1},
		\end{equation}
	\begin{equation}\label{eq:closure-integral}
			I_{k-1}=b\frac{n-k+1}{k-1}I_{k-2},
	\end{equation}
		where $q=(n-k)/k$, then $\partial\Omega$ is a round sphere of radius
		$b^{-1}$.
	\end{proposition}
	
	\begin{proof}
		Multiply \eqref{eq:closure-pointwise} by $h>0$ and integrate.  By
		\eqref{eq:minkowski} and \eqref{eq:closure-integral},
		\begin{align*}
			\int_{\partial\Omega}hH_k\,dS
			&=qI_{k-1},\\
			qb\int_{\partial\Omega}hH_{k-1}\,dS
			&=qb\frac{n-k+1}{k-1}I_{k-2}=qI_{k-1}.
		\end{align*}
		Thus the continuous nonnegative defect has zero weighted integral, and
		\begin{equation}\label{eq:H-equality-final}
			H_k=qbH_{k-1}\qquad\text{on }\partial\Omega.
		\end{equation}
		Since
		\[
		\frac{\binom{n-1}{k}}{\binom{n-1}{k-1}}=q,
		\]
		this is $\cH_k=b\cH_{k-1}$.  The normalized Minkowski formulas give
		\begin{equation}\label{eq:normalized-integral-final}
			\int_{\partial\Omega}\cH_{k-1}\,dS
			=b\int_{\partial\Omega}\cH_{k-2}\,dS.
		\end{equation}
		Newton--Maclaurin on $\Gamma_k$ yields
		\[
		\cH_{k-1}^2\ge\cH_k\cH_{k-2}
		=b\cH_{k-1}\cH_{k-2},
		\]
		and hence $\cH_{k-1}\ge b\cH_{k-2}$.  Its integral defect is zero by
		\eqref{eq:normalized-integral-final}; equality therefore holds
		pointwise in Newton--Maclaurin.  The equality case gives
		\[
		\kappa_1=\cdots=\kappa_{n-1}.
		\]
		A connected compact totally umbilical Euclidean hypersurface is a
		round sphere.  For a sphere of radius $R$,
		$H_k/H_{k-1}=q/R$; comparison with
		\eqref{eq:H-equality-final} gives $R=b^{-1}$.
	\end{proof}

	\section{The proof for the case \texorpdfstring{$2k<n$}{2k<n}}
	\label{sec:proof-subcritical}
	
	\noindent\textit{Proof of \cref{thm:subcritical}.}
	
	Assume in this section that $q>1$, equivalently $2k<n$.
	
	\medskip
	\noindent\textbf{Step 1: normalization, the sharp gradient bound, and the
		boundary curvature inequality.}
	
	Set
	\[
	q=\frac{n-k}{k},\qquad
	U=\frac{u+1}{q-1},\qquad
	v=(-u)^{-1/(q-1)},\qquad
	b=\frac{c}{q-1}=\frac{kc}{n-2k}.
	\]
	Then $U=F_q(v)$, $v=1$, and $|Dv|=b$ on $\partial\Omega$, while
	\[
	M_q[v]=vD^2v-qDv\otimes Dv\in\overline\Gamma_k,
	\qquad \sigma_k(M_q[v])=0\quad\text{in }E.
	\]
	The subcritical exterior approximation and Lemma
	\ref{thm:gradient} give $|Dv|\le b$ in $E$.  Applying the boundary
	viscosity contact argument in Proposition \ref{prop:boundary-curvature} yields
	\begin{equation}\label{eq:subcritical-boundary-inequality}
		H_k\ge qbH_{k-1}>0\qquad\text{on }\partial\Omega,
	\end{equation}
	hence $\partial\Omega$ is strictly k-convex.
	
	We now return to the original capacitary solution $u$ and put
	\begin{equation}\label{eq:I-J-sub}
		J=\int_E\sigma_{k-1}(D^2u)|Du|^2\,dx.
	\end{equation}
	The asymptotic \eqref{eq:package-estimates-revised} makes $J$ finite and makes all
	outer-sphere terms below tend to zero.
	
	\medskip
	\noindent\textbf{Step 2: the first exterior Rellich identity.}
	
	\begin{lemma}[First exterior identity]\label{lem:first-subcritical}
		One has
		\begin{equation}\label{eq:first-subcritical}
			(k+1)J+c^{k+1}I_{k-2}-2c^kI_{k-1}=0.
		\end{equation}
	\end{lemma}
	
	\begin{proof}
This is precisely the integral identity proved in \cite[Lemma~9]{YinZhou}. 
\end{proof}
%		Let
%		\[
%		X_i=T_{k-2}^{ij}(D^2u)u_j|Du|^2.
%		\]
%		The divergence-free property and \eqref{eq:newton-recursion} give
		%\begin{equation}\label{eq:div-X-sub}
%			\diver X
%			=(k+1)\sigma_{k-1}(D^2u)|Du|^2
%			-2T_{k-1}(D^2u)[Du,Du].
%		\end{equation}
%		Also,
		%\begin{equation}\label{eq:div-uTdu-sub}
	%		\diver\bigl(uT_{k-1}(D^2u)Du\bigr)
%			=T_{k-1}(D^2u)[Du,Du],
%		\end{equation}
%		because $\sigma_k(D^2u)=0$.  The outward normal of the truncated
%		exterior domain on $\partial\Omega$ is $-\nu$.  Since
%		$u=-1$ and $Du=c\nu$ there,
%		\begin{align}
%			T_{k-1}(D^2u)[\nu,\nu]&=c^{k-1}H_{k-1},
%			\label{eq:Tnormal-km1}\\
%			T_{k-2}(D^2u)[\nu,\nu]&=c^{k-2}H_{k-2}.
%			\label{eq:Tnormal-km2}
%		\end{align}
%		Integrating \eqref{eq:div-uTdu-sub} yields
%		\[
%		\int_ET_{k-1}(D^2u)[Du,Du],dx=c^kI_{k-1}.
%		\]
%		The inner flux of $X$ is $-c^{k+1}H_{k-2}$.  Integration of
%		\eqref{eq:div-X-sub} proves \eqref{eq:first-subcritical}.
%	\end{proof}
	
	\medskip
	\noindent\textbf{Step 3: the Rellich-Pohozaev-type identity.}
	
	\begin{lemma}[Exterior PRellich-Pohozaev-type identity]\label{lem:pohozaev-subcritical}
		One has
		\begin{equation}\label{eq:pohozaev-support-sub}
			c^{k+1}\int_{\partial\Omega}hH_{k-1}\,dS
			+(n-k+1)J
			-2c^k\int_{\partial\Omega}hH_k\,dS=0.
		\end{equation}
		Equivalently,
		\begin{equation}\label{eq:pohozaev-sub}
			(n-k+1)\left(J+\frac{c^{k+1}}{k-1}I_{k-2}\right)
			-\frac{2(n-k)c^k}{k}I_{k-1}=0.
		\end{equation}
	\end{lemma}
	
	\begin{proof}
Identity \eqref{eq:pohozaev-sub} is exactly \cite[Lemma~10]{YinZhou}. The proof in the cited paper is carried out for the smooth nondegenerate approximations and then passed to the \(C^{1,1}\) limit, so no classical boundary trace of \(D^2u\) is required. Finally, the Hsiung--Minkowski formulas \[ \int_{\partial\Omega}hH_{k-1}\,dS =\frac{n-k+1}{k-1}I_{k-2}, \qquad \int_{\partial\Omega}hH_k\,dS =\frac{n-k}{k}I_{k-1} \] show that \eqref{eq:pohozaev-sub} and \eqref{eq:pohozaev-support-sub} are equivalent.

    \end{proof}
	%	Write $T=T_{k-1}(D^2u)$ and set
		%\begin{equation}\label{eq:Y-sub}
			%Y_i=2uT^{ij}u_{j\ell}x_\ell-T^{i\ell}|Du|^2x_\ell.
	%	\end{equation}
	%	Using $\partial_iT^{ij}=0$, commutation of $T$ and $D^2u$, and
		%$T:D^2u=k\sigma_k(D^2u)=0$, direct differentiation gives
	%	\begin{equation}\label{eq:div-Y-sub}
	%		\diver Y=-(n-k+1)\sigma_{k-1}(D^2u)|Du|^2.
	%	\end{equation}
	%	On $\partial\Omega$, constancy of $|Du|$ gives
	%	$u_{\alpha\nu}=0$, while
		%$u_{\alpha\beta}=c\kappa_\alpha\delta_{\alpha\beta}$.  The boundary
	%	equation is
		%\begin{equation}\label{eq:boundary-equation-u}
	%		c^kH_k+c^{k-1}H_{k-1}u_{\nu\nu}=0.
	%	\end{equation}
	%	Consequently
	%	\[
	%	Y\cdot(-\nu)=-2c^khH_k+c^{k+1}hH_{k-1}.
	%	\]
	%	Integrating \eqref{eq:div-Y-sub} gives
		%\eqref{eq:pohozaev-support-sub}; \eqref{eq:minkowski} with
	%	$j=k-1,k$ gives \eqref{eq:pohozaev-sub}.

	Eliminating $J$ from
	\eqref{eq:first-subcritical} and \eqref{eq:pohozaev-sub} gives the
	identity needed later.
	
	\medskip
	\noindent\textbf{Step 4: elimination of the bulk integral.}
	
	\begin{proposition}[Subcritical boundary ratio]\label{prop:ratio-sub}
		In the normalization \eqref{eq:normal-sub},
		\begin{equation}\label{eq:universal-ratio-sub}
			I_{k-1}=b\frac{n-k+1}{k-1}I_{k-2}.
		\end{equation}
	\end{proposition}
	
	\begin{proof}
		Solving \eqref{eq:first-subcritical} for $J$ and substituting into
		\eqref{eq:pohozaev-sub} yields
		\[
		c=\frac{n-2k}{k}\frac{k-1}{n-k+1}
		\frac{I_{k-1}}{I_{k-2}}.
		\]
		Since $b=kc/(n-2k)$, this is \eqref{eq:universal-ratio-sub}.
	\end{proof}
	
	\medskip
	\noindent\textbf{Step 5: geometric rigidity and identification of the solution.}
	
	%Strict $k$-convexity follows from \eqref{eq:subcritical-boundary-inequality}.  
    By Proposition
	\ref{prop:boundary-curvature},
	\[
	H_k\ge qbH_{k-1},
	\]
	and Proposition \ref{prop:ratio-sub} supplies the matching integral ratio
	\eqref{eq:closure-integral}.  Therefore Proposition
	\ref{prop:geometric-closure} gives
	$\Omega=B_R(x_0)$ with
	\[
	R=b^{-1}=\frac{n-2k}{kc}.
	\]
	The radial function
	\[
	u_0(x)=-\left(\frac{R}{|x-x_0|}\right)^{(n-2k)/k}
	\]
	is $k$-admissible, has the prescribed boundary value and decay, and
	solves $\sigma_k(D^2u_0)=0$ in the exterior of the ball.  Uniqueness in
	the capacitary class therefore gives $u=u_0$.
	\hfill$\Box$
	
	\section{The proof for the case \texorpdfstring{$2k=n$}{2k=n}}
	\label{sec:proof-critical}
	
	\noindent\textit{Proof of \cref{thm:critical}.}
	
	Assume now that $q=1$.
	
	\medskip
	\noindent\textbf{Step 1: normalization, exterior approximation, and the
		finite part at infinity.}
	
	Set
	\[
	U=\frac{u}{\mathcal{A}},\qquad v=e^U,\qquad b=\frac {c}{\mathcal{A}}.
	\]
	Then $U=\log v$, $v=1$, and $|Dv|=b$ on $\partial\Omega$, and
	\[
	M_1[v]=vD^2v-Dv\otimes Dv\in\overline\Gamma_k,
	\qquad \sigma_k(M_1[v])=0.
	\]
	The gauge exhaustion and the nondegenerate construction in Proposition
	\ref{prop:gauge-exhaustion} and Proposition \ref{prop:star-package} apply because the
	boundary is strictly star-shaped and strictly $(k-1)$-convex.  Thus
	$U$ is the local $C^1$ limit of the smooth admissible approximants $U_{\varepsilon} $ satisfying
	\begin{equation}
		\sigma_k (D^2U_{\varepsilon})=F_{\varepsilon},
	\end{equation}
where $F_{\varepsilon}=c_{n,k,q}\varepsilon^2(|x|^2+\varepsilon^2)^{-n/2-1}.$
And for any compact $A\subset \bar{E},$ 
there is 
\begin{equation}
	U_{\varepsilon}\rightarrow U\quad\text{and}\quad DU_{\varepsilon}\rightarrow DU~~\text{uniformly in}~A.
\end{equation}
 By \cref{prop:finite-part}, for
	some $\gamma\in\R$,
	\begin{equation}\label{eq:critical-finite-part-step}
		U(x)-\log(b|x|)\longrightarrow\gamma,\qquad
		DU(x)-\frac{x}{|x|^2}=o(|x|^{-1}).
	\end{equation}

	\medskip
	\noindent\textbf{Step 2: the sharp gradient bound and boundary curvature
		inequality.}
	
 According to Lemma \ref{thm:gradient},	 we have $|Dv|\le b$ in $E$.  Then by Proposition \ref{prop:boundary-curvature}, there is
	\begin{equation}\label{eq:critical-boundary-inequality}
		H_k\ge bH_{k-1}\qquad\text{on }\partial\Omega.
	\end{equation}
	Since $H_{k-1}>0$, the boundary is automatically strictly
	$k$-convex.  It remains to derive the integral identity \eqref{eq:closure-integral} by a
	scale-invariant flux.
	
	\medskip
	\noindent\textbf{Step 3: vanishing of the %Wronskian
    flux at infinity.}

%	The following formulas involve only $U$, $DU$, and Newton tensors and
%	are therefore meaningful distributionally.
	We shall use the following weak-flux convention. If
$X\in L^\infty_{\mathrm{loc}}$ and $\diver X=0$ in distributions on an
annulus $A_{r_-,r_+}=B_{r_+}\setminus B_{r_-}\subset \mathbb{R}^n\setminus\bar{\Omega}.$, choose a smooth radial function $\phi$ which is
one near $\partial B_{r_-}$ and zero near $\partial B_{r_+}$, we define the weak normal flux of $X$ by
\begin{equation}\label{eq:weak-flux}
    F (X):=-\int_{A_{r_-,r_+}}X\cdot D\phi\,dx.
\end{equation}
This definition is independent of $\phi$ and $A_{r_-,r_+}$.
%For a smooth field it is exactly
%$\int_{\partial B_r}X\cdot\nu_r\,dS$. 
If $X$ is smooth, we have
\begin{align*}
    0=\int_{A_{r_-,r_+}}(\diver X) \phi=-\int_{A_{r_-,r_+}}X\cdot D\phi dx-\int_{\partial B_{r_-}}X\cdot \nu_{r_-}ds,
\end{align*}
then 
\begin{equation}\nonumber
    F(X)=\int_{\partial B_{r_-}}X\cdot \nu_{r_-}ds=\int_{\partial B_{r}}X\cdot \nu_{r}ds,\quad\text{for}~r_-<r<r_+.
\end{equation}
Thus \eqref{eq:weak-flux} is the weak normal
flux through any separating sphere.

\begin{lemma}\label{lem:critical-currents}
			Denote
		\[
		T=T_{k-1}(D^2U),\qquad \widehat T=T_k(D^2U),
		\qquad Z=x\cdot DU-1.
		\]
		Define
		\begin{align}
			\cK[U]&:=T_{k-1}(D^2U)DU-T_k(D^2U) x,\label{eq:K-critical}\\
			\cJ[U]&:=(U-x\cdot DU+1)T_{k-1}(D^2U)DU-UT_{k}(D^2U) x
			\label{eq:J-critical},
		\end{align}
	then 
	\begin{equation}\nonumber
		\diver \cK[U]=0,\quad \diver \cJ[U]=0\quad\text{in}~\mathcal{D}'(E).
	\end{equation}
	  Moreover,
	  there is 
	 % $$\int_{\partial B_s}^{}\cK \cdot \nu_{s}=0,$$
     $$F(\cK[U])=0.$$
	\end{lemma}
	
	\begin{proof}
		Since $U\in C_{loc}^{1,1}(\bar{E}),$ we first perform the calculation for $U_{\varepsilon}.$
		We denote 
		$$\cK_{\varepsilon}:=\cK[U_{\varepsilon}]=T_{k-1}(D^2U_{\varepsilon})DU_{\varepsilon}-T_{k}(D^2U_{\varepsilon})x,$$
		$$	\cJ_{\varepsilon}:=\cJ[U_{\varepsilon}]=(U_{\varepsilon}-x\cdot DU_{\varepsilon}+1)T_{k-1}(D^2U_{\varepsilon})DU_{\varepsilon}-U_{\varepsilon}T_{k}(D^2U_{\varepsilon}) x.
	$$
		By \eqref{eq:newton-euler}, \eqref{eq:newton-trace} and \eqref{eq:newton-divfree}, we have
		\begin{align}\label{K1}
			\diver (T_{k-1}(D^2U_{\varepsilon})DU_{\varepsilon})
			&=T_{k-1}(D^2U_{\varepsilon}):D^2U_{\varepsilon}
			+\sum_{i,j=1}^{n}(\partial_iT_{k-1}^{ij}(D^2U_{\varepsilon}))\partial_jU_{\varepsilon}\notag\\
			&=k\sigma_k(D^2U_{\varepsilon})\notag\\
			&=kf_{\varepsilon},
		\end{align}
	and
\begin{align}\label{K2}
			\diver (T_k(D^2U_{\varepsilon})x)&=Tr(T_k(D^2U_{\varepsilon}))+\sum_{i,j=1}^{n}(\partial_iT_{k}^{ij}(D^2U_{\varepsilon}))x_j\notag\\
			&=(n-k)\sigma_k(D^2U_{\varepsilon})\notag\\
			&=(n-k)f_{\varepsilon}.
		\end{align}
	Combing \eqref{K1} and \eqref{K2}, there is
		\begin{equation}\label{divK}
			\diver \cK_{\varepsilon}=(2k-n)f_{\varepsilon}=0.
		\end{equation}
	By \eqref{eq:newton-recursion} and \eqref{K1}, we have
	\begin{align}\label{J1}
		&\diver [(U_{\varepsilon}-x\cdot DU_{\varepsilon}+1)T_{k-1}(D^2U_{\varepsilon})DU_{\varepsilon}]\notag\\
		=&D(U_{\varepsilon}-x\cdot DU_{\varepsilon}+1)\cdot T_{k-1}(D^2U_{\varepsilon})DU_{\varepsilon}+ (U_{\varepsilon}-x\cdot DU_{\varepsilon}+1)\diver (T_{k-1}(D^2U_{\varepsilon})DU_{\varepsilon})\notag\\
		=&-x^TD^2U_{\varepsilon} T_{k-1}(D^2U_{\varepsilon})DU_{\varepsilon}+kf_{\varepsilon}(U_{\varepsilon}-x\cdot DU_{\varepsilon}+1)\notag\\
		=&x^T(T_{k}(D^2U_{\varepsilon})-\sigma_k(D^2U_{\varepsilon})I)DU_{\varepsilon}+kf_{\varepsilon}(U_{\varepsilon}-x\cdot DU_{\varepsilon}+1)\notag\\
		=& x\cdot T_{k}(D^2U_{\varepsilon})DU_{\varepsilon}+f_{\varepsilon}[k(U_{\varepsilon}-x\cdot DU_{\varepsilon}+1)-x\cdot DU_{\varepsilon
		}].
	\end{align}	 
By \eqref{K2}, there is
\begin{align}\label{J2}
			\diver [U_{\varepsilon}T_k(D^2U_{\varepsilon
			})x]
		&=DU_{\varepsilon}\cdot T_k(D^2U_{\varepsilon
		})x+U_{\varepsilon}\diver[T_k(D^2U_{\varepsilon
	})x]\notag\\
&=DU_{\varepsilon}\cdot T_k(D^2U_{\varepsilon
})x+(n-k)U_{\varepsilon}f_{\varepsilon}\notag\\
&=x\cdot T_k(D^2U_{\varepsilon
})DU_{\varepsilon}+(n-k)U_{\varepsilon}f_{\varepsilon}.
		\end{align}
Combining \eqref{J1}  and \eqref{J2}, we  obtain that 
\begin{equation}\label{eq:Jeps-error-critical}
	\diver \cJ_{\varepsilon}=f_{\varepsilon}[-(k+1)x\cdot DU_{\varepsilon}+k].
\end{equation}
According to Lemma \ref{lem:weak-newton},
		\begin{equation}\nonumber
T_k(D^2U_{\varepsilon})\stackrel{\ast}{\rightharpoonup}T_k(D^2U),\quad T_{k-1}(D^2U_{\varepsilon})\stackrel{\ast}{\rightharpoonup}T_{k-1}(D^2U)
			\quad\text{in }L^\infty_{\mathrm{loc}}.
        \end{equation}
Since $U_\varepsilon\to U$ and $DU_\varepsilon\to DU$ locally
uniformly, multiplication by these strongly convergent factors gives
\[
    \cK_\varepsilon\stackrel{*}{\rightharpoonup}\cK[U],
    \qquad
    \cJ_\varepsilon\stackrel{*}{\rightharpoonup}\cJ[U]
    \quad\text{in }L^\infty_{\mathrm{loc}}.
\]
Hence, for every $\varphi\in C_c^\infty(E)$,
\begin{align*}
\langle\diver \cK[U],\varphi\rangle
&=-\int_E\cK[U]\cdot D\varphi
  =\lim_{\varepsilon\to0}
    \int_E\varphi\,\diver K_\varepsilon=0,\\
\langle\diver\cJ[U],\varphi\rangle
&=-\int_E\cJ[U]\cdot D\varphi
  =\lim_{\varepsilon\to0}
    \int_E\varphi\,\diver\cJ_\varepsilon=0.
\end{align*}		
The second limit follows from \eqref{eq:Jeps-error-critical}, because
$x\cdot DU_\varepsilon$ is uniformly bounded on the support of
$\varphi$ and $f_\varepsilon\to0$ in $L^1_{\mathrm{loc}}$. This proves
the two distributional identities.	
%		By \eqref{eq:newton-divfree}--\eqref{eq:newton-trace},
%		\[
		%\diver(TDU)=k\sigma_k(D^2U)=0,
%		\qquad
	%	\diver(\widehat T x)=(n-k)\sigma_k(D^2U)=0.
%		\]
%		Thus $\diver\cK=0$.  Since $TD^2U=-\widehat T$, the smooth form is
%		$\cK=T DZ$.
		
%		For $a=U-x\cdot DU+1$, $Da=-D^2U\,x$.  Hence
%		\[
		%\diver(aTDU)=x\cdot\widehat TDU
	%	=\diver(U\widehat T x),
	%	\]
%		which proves $\diver\cJ=0$.  These identities pass to $C^{1,1}$ by
	%	\cref{lem:weak-newton}.
Define 
\begin{equation}\label{def of UR}
    U_R(x):=U(Rx)-\log R,
\end{equation}
then %by the definition of $\cK,$ 
\begin{equation}\nonumber
    \cK[U_R](x)=R^{2k-1}\cK[U](Rx),
\end{equation}
and for $R>1,$ there is 
\begin{align}
    F(\cK[U_R])
    &=-\int_{B_{r_+}\setminus B_{r_-}}\cK[U_R]\cdot D_x\phi dx\notag\\
    &=-\int_{B_{r_+}\setminus B_{r_-}} R^{2k-1}\cK[U](Rx)\cdot D_x\phi dx\notag\\
    &=-\int_{B_{Rr_+}\setminus B_{Rr_-}} R^{2k-1}\cK[U](y)\cdot RD_y\phi(\frac{y}{R}) R^{-n}dy\notag\\
    &=-\int_{B_{Rr_+}\setminus B_{Rr_-}}R^{2k-n}\cK[U](y)D_y\phi(\frac{y}{R})dy .\notag
\end{align}
Since $k=\frac{n}{2}$ and $F(\cK[U])$ is independent of $\phi$ and the annulus domain,
then
$$F(\cK[U_R])=-\int_{B_{Rr_+}\setminus B_{Rr_-}}R^{2k-n}\cK[U](y)D_y\phi(\frac{y}{R})dy=F(\cK[U]).$$
According to Proposition \ref{prop:finite-part},
\begin{equation}\nonumber
    U_R(x)\to
    \Psi(x):=\log(|x|)+\log b+\gamma\quad~\text{as}~R\to\infty,
\end{equation}
locally uniformly in $\mathbb{R}^n\setminus\{0\}$, with a uniform $C^{1,1}$ bound on every fixed
annulus. Lemma \ref{lem:weak-newton} and the strong convergence of the gradients yield
\[
    K[U_R]\stackrel{*}{\rightharpoonup}K[\Psi]
    \quad\text{in }L^\infty_{loc}.
\]
Note that 
$$\sigma_k(D^2\Psi)=0,$$
\begin{equation}\nonumber
   x\cdot D\Psi-1=0,
\end{equation}
and
$$0=D( x\cdot D\Psi-1)=D\Psi+D^2\Psi x,$$
then 
\begin{align}\label{KPsi}
    0=T_{k-1}(D^2\Psi)D( x\cdot D\Psi-1)
    &=T_{k-1}(D^2\Psi)D\Psi+T_{k-1}(D^2\Psi)D^2\Psi x\notag\\
    &=T_{k-1}(D^2\Psi)D\Psi-T_k(D^2\Psi)x+\sigma_k(D^2\Psi)Ix\notag\\
    &=T_{k-1}(D^2\Psi)D\Psi-T_k(D^2\Psi)x\notag\\
    &= \cK[\Psi]\quad a.e.~\text{in} ~E.
\end{align}
This implies that
\begin{equation}\nonumber
    F(\cK[U])=F(\cK[U_R])=F(\cK[\Psi])=0.
\end{equation}
This proves the
flux assertion without requiring a classical trace of $D^2U$ on $\partial\Omega.$

%		On $\partial\Omega$, $U=0$, $DU=b\nu$, and tangential
%		differentiation gives
%		\[
%		T^{\nu\nu}=b^{k-1}H_{k-1},\qquad
%		\widehat T^{\nu\nu}=b^kH_k,
%		\qquad \widehat T^{\alpha\nu}=0.
%		\]
%		Therefore
%		\[
		%\cK\cdot\nu=b^kH_{k-1}-b^khH_k.
		%\]
		%Because $n=2k$, \eqref{eq:minkowski} gives
		%$\int hH_k=\int H_{k-1}$.  The inner flux vanishes, and conservation
	%	shows that every outer flux vanishes as well.
	\end{proof}
	
	%\medskip
	%\noindent\textbf{Step 4: vanishing of the Wronskian flux at infinity.}

\begin{lemma}[Vanishing Wronskian flux]\label{lem:vanishing-flux}
For every sufficiently large $R$, the weak normal trace of $\cJ$ on
$\partial B_R$ satisfies
\begin{equation}\nonumber
  %  \int_{\partial B_R}\cJ\cdot\nu_R\,dS=0.
F(\cJ[U])=0.
\end{equation}
%Here and below the surface integral is understood in the weak-flux
%sense of \eqref{eq:weak-flux} whenever a classical trace is unavailable.
\end{lemma}

	\begin{proof}
		The flux is independent of the annulus $A_{r_-,r_+}=B_{r_+}\setminus B_{r_-}\subset E$  and the smooth radial function $\phi$ by Lemma \ref{lem:critical-currents}.
        Define the function $U_R$ as in \eqref{def of UR}.
		%Consider the blow-down
		%\[
		%U_R(y)=U(Ry)-\log R.
		%\]
		By Proposition \ref{prop:finite-part}, $U_R$ converges uniformly on fixed annuli
		to
		\[
		\Phi(y)=\log|y|+\gamma+\log b,
		\]
		and the rescaled functions are uniformly $C^{1,1}$. 
        According to \eqref{eq:J-critical},
        \begin{equation}\nonumber
             \cJ[U_R]=\cJ[U(Rx)-\log R] =R^{2k-1}\cJ [U](Rx)-(\log R)\cK[U(Rx)].
        \end{equation}
       Then 
    \begin{align}
       F(\cJ[U_R])
    &=-\int_{B_{r_+}\setminus B_{r_-}}\cJ[U_R]\cdot D_x\phi dx\notag\\
    &=-\int_{B_{r_+}\setminus B_{r_-}} R^{2k-1}\cJ[U](Rx)\cdot D_x\phi dx+(\log R) F(\cK[U_R])\notag\\
    &=-\int_{B_{Rr_+}\setminus B_{Rr_-}} R^{2k-1}\cJ[U](y)\cdot RD_y\phi(\frac{y}{R}) R^{-n}dy\notag\\
    &=-\int_{B_{Rr_+}\setminus B_{Rr_-}}R^{2k-n}\cJ[U](y)D_y\phi(\frac{y}{R})dy .\notag
    \end{align}
        Since $k=\frac{n}{2}$ and $F(\cJ[U])$ is independent of $\phi$ and the annulus domain,
then
\begin{equation}\label{FJUR}
    F(\cJ[U_R])=-\int_{B_{Rr_+}\setminus B_{Rr_-}}R^{2k-n}\cJ[U](y)D_y\phi(\frac{y}{R})dy=F(\cJ[U]).
\end{equation}
Let $\Psi(x)$ be defined as in Lemma \ref{lem:critical-currents}, then by Proposition \ref{prop:finite-part} and Lemma \ref{lem:weak-newton}, there is
\begin{equation}\label{URPsi}
    \cJ[U_R]\stackrel{\ast}{\rightharpoonup}\cJ[\Psi]\quad\text{in}~L^{\infty}_{loc}.
\end{equation}

Note that 
$$x\cdot D\Psi-1=0,$$
then by \eqref{KPsi}, we obtain that
    \begin{equation}\label{JPsi}
            \cJ[\Psi]=\Psi\cK[\Psi]=0\quad a.e.~\text{in} ~E.
     \end{equation}
 Therefore, combining \eqref{FJUR}, \eqref{URPsi} and \eqref{JPsi}, we prove that
\begin{equation}\nonumber
    F(\cJ[U])=F(\cJ[U_R])=F(\cJ[\Psi])=0.
\end{equation}       
	%	Choose a radial cutoff on a fixed annulus.  By
	%	\cref{lem:weak-newton}, its pairing with $\cJ[U_R]$ converges to the
	%	pairing with $\cJ[\Phi]$.  For $\Phi$, the defect
	%	$y\cdot D\Phi-1$ vanishes and the Wronskian current is zero.  The
	%	constant-in-$R$ flux therefore tends to zero and must vanish.
	\end{proof}
	
	\medskip
	\noindent\textbf{Step 4: the critical boundary ratio.}
	
	\begin{proposition}[Critical boundary ratio]\label{prop:ratio-critical}
		Let 
        $$I_j=\int_{\partial\Omega}H_jds,$$
		then
        \begin{equation}\label{eq:universal-ratio-critical}
			I_{k-1}=b\frac{n-k+1}{k-1}I_{k-2}.
		\end{equation}
	\end{proposition}
	
	\begin{proof}
We first compute the inner-boundary flux for the smooth nondegenerate
approximants. Put
\[
    \beta_\varepsilon=(U_\varepsilon)_\nu
    \quad\text{on }\partial\Omega.
\]
Since $U_\varepsilon=0$ on $\partial\Omega$,
\[
    DU_\varepsilon=\beta_\varepsilon\nu,\qquad
    (T_\varepsilon)^{\nu\nu}
      =\beta_\varepsilon^{k-1}H_{k-1}.
\]
Also
\[
    a_\varepsilon
      =U_\varepsilon-x\cdot DU_\varepsilon+1
      =1-\beta_\varepsilon h.
\]
Because the exterior inner normal is $-\nu$ and the second term in
$\cJ_\varepsilon$ vanishes on $\partial\Omega$, we obtain the exact
smooth boundary identity
\[
\begin{aligned}
    \cJ_\varepsilon\cdot(-\nu)
    &=-a_\varepsilon\beta_\varepsilon
       (T_\varepsilon)^{\nu\nu}\\
    &=\beta_\varepsilon^kH_{k-1}
       (\beta_\varepsilon h-1).
\end{aligned}
\]

To pass this identity to the limit without assuming a classical
Hessian trace, choose a smooth cutoff $\eta$ on $\overline E$ which is
one near $\partial\Omega$ and zero outside a large ball. The classical
Gauss--Green formula for $\cJ_\varepsilon$ gives
\[
\int_{\partial\Omega}
   \cJ_\varepsilon\cdot(-\nu)\,ds
=\int_E\eta\,\diver\cJ_\varepsilon\,dx
 +\int_E\cJ_\varepsilon\cdot D\eta\,dx.
\]
By \eqref{eq:Jeps-error-critical}, the first term on the right tends
to zero. By the weak-star convergence
$\cJ_\varepsilon\stackrel{\ast}{\rightharpoonup}\cJ$, we have
\begin{equation}\nonumber
    \int_E\cJ_\varepsilon\cdot D\eta\,dx\to \int_E\cJ\cdot D\eta\,dx\quad\text{as}~\varepsilon\to 0.
\end{equation}
On the other hand,
$\beta_\varepsilon\to b$ uniformly on $\partial\Omega$, and hence

\[
  \int_{\partial\Omega}
   \cJ_\varepsilon\cdot(-\nu)\,ds\to  \int_{\partial\Omega}
       b^kH_{k-1}(bh-1)\,ds\quad \text{as}~\varepsilon\to 0.
\]
This is precisely the weak inner flux of the
limiting current $\cJ$.

By
Lemma~\ref{lem:vanishing-flux}, there is
\[ \int
_E\cJ\cdot D\eta dx=0,\]
this implies that 
\begin{equation}\nonumber
    \int_{\partial\Omega}
       b^kH_{k-1}(bh-1)\,ds=0.
\end{equation}
Since $b>0$,
\begin{equation}\label{Ik-1}
b\int_{\partial\Omega}hH_{k-1}\,ds
       =\int_{\partial\Omega}H_{k-1}\,ds
       =I_{k-1}.
\end{equation}
Finally, by \eqref{eq:minkowski}, there is 
\begin{equation}\label{hHk-1}
   \int_{\partial\Omega}hH_{k-1}ds=\frac{n-k+1}{k-1} \int_{\partial\Omega}H_{k-2}ds=\frac{n-k+1}{k-1}I_{k-2} .
\end{equation}
Combining \eqref{Ik-1} and \eqref{hHk-1}, we finally obtain
\eqref{eq:universal-ratio-critical}.
\end{proof}

	\medskip
	\noindent\textbf{Step 5: geometric rigidity and identification of the solution.}
	
	By \cref{prop:boundary-curvature},
	$H_k\ge bH_{k-1}>0$ because $q=1$ and the boundary is strictly
	$(k-1)$-convex.  Thus the boundary is in fact strictly $k$-convex.
	Together with \cref{prop:ratio-critical}, both hypotheses of
	\cref{prop:geometric-closure} hold.  Hence
	\[
	\Omega=B_R(x_0),\qquad R=b^{-1}=\frac {\mathcal{A}}{c}.
	\]
	The function
	\[
	u_0(x)=\mathcal{A}\log\frac{|x-x_0|}{R}
	\]
	solves the homogeneous $k$-Hessian equation in the exterior of this
	ball, vanishes on its boundary, and has the prescribed logarithmic
	coefficient.  Uniqueness in the logarithmic class gives $u=u_0$.
	\hfill$\Box$
	
	\section{The proof for the case \texorpdfstring{$2k>n$}{2k>n}}
	\label{sec:proof-supercritical}
	
	\noindent\textit{Proof of \cref{thm:supercritical}.}
	
	Assume $0<q<1$ and set $\alpha=1-q$.
	
	\medskip
	\noindent\textbf{Step 1: normalization and the pointwise boundary inequality.}
	
	Retain the constants
	\[
	R_\ast=\left(\frac{\mathfrak a\alpha}{c}\right)^{1/q},\qquad
	\beta=\mathfrak a R_\ast^\alpha,\qquad b=R_\ast^{-1},
	\]
	and define
	\[
	U=\frac{u-1}{\alpha\beta},\qquad
	v=\left(\frac{u-1+\beta}{\beta}\right)^{1/\alpha}.
	\]
	Then $U=F_q(v)$, $v=1$, and $|Dv|=b$ on $\partial\Omega$, and
	$M_q[v]$ satisfies the homogeneous transformed equation.  The
	star-shaped exterior package and Proposition \ref{prop:finite-part} give a
	constant $C_0$ such that
	\begin{equation}\label{eq:super-finite-part-step}
		u(x)-1+\beta=\mathfrak a|x|^\alpha+C_0+o(1).
	\end{equation}
	The common gradient estimate gives $|Dv|\le b$, and the boundary
	contact argument gives
	\begin{equation}\label{eq:super-boundary-inequality}
		H_k\ge qbH_{k-1}\qquad\text{on }\partial\Omega.
	\end{equation}
	In particular, strict $(k-1)$-convexity improves to strict
	$k$-convexity.

	\medskip
	\noindent\textbf{Step 2: the Newton--Jacobi current and its inner flux.}
	
	Define
	\begin{equation}\label{eq:w-super}
		w=u-1+\beta=\beta v^\alpha.
	\end{equation}
	Then $\sigma_k(D^2w)=0$, while on $\partial\Omega$,
	\begin{equation}\label{eq:w-boundary-super}
		w=\beta,\qquad Dw=c\nu,\qquad c=\alpha\beta b.
	\end{equation}
	Let $\mathcal{T}=T_{k-1}(D^2w)$ and set
	\begin{equation}\label{eq:Z-super}
		Z^w=x\cdot Dw-\alpha w.
	\end{equation}
    By Euler’s identity, there is 
    $$\mathcal{T}^{ij}w_{ij}=k\sigma_k(D^2w)=0,$$
    this implies that
    \begin{align}
        \mathcal{T}^{ij}Z^w_{ij}
        &=\mathcal{T}^{ij}\left((1+q)w_{ij}+x_lw_{ijl}\right)\notag\\
        &=(1+q)\mathcal{T}^{ij}w_{ij}+x_l(\sigma_k(D^2w))_l\notag\\
        &=0.\notag
    \end{align}
    Let $w_{\varepsilon
    }=u_{\varepsilon}-1+\beta
    $ be the smooth nondegenerate approximations from Proposition \ref{prop:star-package},
	and set 
	\begin{equation}\nonumber
		\sigma_k (D^2w_{\varepsilon})=\mathfrak a^kF_{\varepsilon},\quad (w_{\varepsilon})_{\nu}|_{\partial\Omega}=g_{\varepsilon},
	\end{equation}
where $F_{\varepsilon}=c_{n,k,q}\varepsilon^2(|x|^2+\varepsilon^2)^{-n/2-1}>0,$ and $g_{\varepsilon}\to c$. For any compact $A\subset \bar{E},$ 
there is 
\begin{equation}\nonumber
	w_{\varepsilon}\rightarrow w\quad\text{and}\quad Dw_{\varepsilon}\rightarrow Dw~~\text{uniformly in}~A.
\end{equation}
%Hence the Newton--Jacobi current
	%\begin{equation}\label{eq:J-super}
	%	\cJ^j=T^{ij}(wZ_i-Zw_i)
	%\end{equation}
	%is divergence-free.  Its outer flux
	%\begin{equation}\label{eq:mass-super}
	%	\mathfrak m=\int_{\partial B_R}\cJ\cdot\nu_R\,dS
	%\end{equation}
	%is independent of $R$.  The same holds for
	%\begin{equation}\label{eq:Q-super}
%		Q=\int_{\partial B_R}T^{ij}w_i(\nu_R)_j\,dS.
%	\end{equation}
%	These identities are first obtained for the smooth approximants and
%	then passed to the limit using the Hessian null-Lagrangian structure.
\begin{lemma}\label{lem:divQKJ}
Define 
$$J^j[w_{\varepsilon}]:=T^{ij}(D^2w_{\varepsilon})(w_{\varepsilon}Z^{w_{\varepsilon}}_{i}-Z^{w_{\varepsilon}}(w_{\varepsilon
    })_{i}),$$
    $$Q^j[w_{\varepsilon}]:=T^{ij}(D^2w_{\varepsilon})(w_{\varepsilon})_{i},$$
    $$ K^j[w_{\varepsilon}]:=T^{ij}(D^2w_{\varepsilon})Z^{w_{\varepsilon}}_i,$$
    then 
    \begin{equation}\label{diverQvar}
        \diver  Q[w_{\varepsilon}]=kF_{\varepsilon},
    \end{equation}\
\begin{equation}\label{diverKvar}
        \diver K[w_{\varepsilon}]=x\cdot DF_{\varepsilon}+k(1+q)F_{\varepsilon},
    \end{equation}
    \begin{equation}\label{diverJvar}
        \diver J[w_{\varepsilon}]=w_{\varepsilon}\left(x\cdot DF_{\varepsilon}+k(1+q)F_{\varepsilon
        }\right)-kZ^{w_{\varepsilon}}F_{\varepsilon}.
    \end{equation}
    And there is
$$\diver Q[w]=\diver K[w]=\diver J[w]=0\quad\text{in}~\mathcal{D}'(E).$$
\end{lemma}	

\begin{proof}
Since
 $T^{ij}(D^2w_{\varepsilon})(w_{\varepsilon})_{ij}=kF_{\varepsilon}$ and $\partial_jT^{ij}(D^2w_{\varepsilon})=0,$
 then we have
 \begin{equation}\nonumber
     \diver Q[w_{\varepsilon}]=(Q^j[w_{\varepsilon}])_j
     =\partial_jT^{ij}(D^2w_{\varepsilon})(w_{\varepsilon})_{i}+T^{ij}(D^2w_{\varepsilon})(w_{\varepsilon})_{ij}
     =kF_{\varepsilon},
 \end{equation}
\begin{align}
    \diver K[w_{\varepsilon}]=(K^j[w_{\varepsilon}])_j
    &=\partial_jT^{ij}(D^2w_{\varepsilon})Z_i^{w_{\varepsilon}}+T^{ij}(D^2w_{\varepsilon
    })Z_{ij}^{w_{\varepsilon}}\notag\\
    &=T^{ij}(D^2w_{\varepsilon
    })\left((1+q)(w_{\varepsilon})_{ij}+x_l(w_{\varepsilon})_{ijl}\right)\notag\\
    &=(1+q)kF_{\varepsilon}+x\cdot DF_{\varepsilon},\notag
\end{align}
and
\begin{align}
    \diver J[w_{\varepsilon}]
    =(J^j[w_{\varepsilon}])_j
    =&\partial_jT^{ij}(D^2w_{\varepsilon})(w_{\varepsilon}Z^{w_{\varepsilon}}_{i}-Z^{w_{\varepsilon}}(w_{\varepsilon})
    _{i})
    +T^{ij}(D^2w_{\varepsilon})((w_{\varepsilon})_jZ^{w_{\varepsilon}}_i-Z_j^{w_{\varepsilon}}(w_{\varepsilon})_i)\notag\\
    &+T^{ij}(D^2w_{\varepsilon})w_{\varepsilon}Z_{ij}^{w_{\varepsilon}}-T^{ij}(D^2w_{\varepsilon})Z^{w_{\varepsilon}}(w_{\varepsilon})_{ij}\notag\\
    =&w_{\varepsilon}[(1+q)kF_{\varepsilon}+x\cdot DF_{\varepsilon}]-kZ^{w_{\varepsilon}}F_{\varepsilon}.\notag
\end{align}

The weak convergence of $K$ and $J$ is not obtained by multiplying two
weakly convergent Hessian factors.  Instead, Newton's recursion gives the
exact identities
\begin{align}
 K[w_\varepsilon]
 &=qQ[w_\varepsilon]+F_\varepsilon x
   -T_k(D^2w_\varepsilon)x,\label{eq:K-weak-rewrite}\\
 J[w_\varepsilon]
 &=w_\varepsilon K[w_\varepsilon]
   -Z^{w_\varepsilon}Q[w_\varepsilon].\label{eq:J-weak-rewrite}
\end{align}
Indeed, $(Z^{w_\varepsilon})_i=q(w_\varepsilon)_i+
 x_\ell(w_\varepsilon)_{\ell i}$ and
$T_{k-1}D^2w_\varepsilon=F_\varepsilon I-T_k$.
By \cref{lem:weak-newton},
\[
 T_j(D^2w_\varepsilon)\stackrel{*}{\rightharpoonup}T_j(D^2w)
 \quad(j=k-1,k),
\]
while $w_\varepsilon,Dw_\varepsilon$, and $Z^{w_\varepsilon}$ converge
locally uniformly and $F_\varepsilon\to0$ locally uniformly.  Equations
\eqref{eq:K-weak-rewrite} and \eqref{eq:J-weak-rewrite} therefore give
\[
 J[w_\varepsilon]\stackrel{*}{\rightharpoonup}J[w],\qquad
 Q[w_\varepsilon]\stackrel{*}{\rightharpoonup}Q[w],\qquad
 K[w_\varepsilon]\stackrel{*}{\rightharpoonup}K[w]
 \quad\text{in }L^\infty_{\mathrm{loc}}.
\]
Hence,  for every $\varphi\in C_c^\infty(E)$, there is
\begin{align*}
\langle\diver Q[w],\varphi\rangle
&=-\int_EQ[w]\cdot D\varphi
  =\lim_{\varepsilon\to0}
    \int_E\varphi\,\diver Q[w_\varepsilon]=0,\\
\langle\diver K[w],\varphi\rangle
&=-\int_EK[w]\cdot D\varphi
  =\lim_{\varepsilon\to0}
    \int_E\varphi\,\diver K[w_\varepsilon]=0,\\
    \langle\diver J[w],\varphi\rangle
&=-\int_EJ[w]\cdot D\varphi
  =\lim_{\varepsilon\to0}
    \int_E\varphi\,\diver J[w_\varepsilon]=0,
\end{align*}		
The second limit follows from \eqref{diverQvar}, \eqref{diverKvar} and \eqref{diverJvar}. This proves the three distributional identities.	
\end{proof}

Choose a smooth cutoff $\phi$ on $E$ which equals one in a collar of $\partial\Omega$, and is zero outside a larger compact set, and has compactly supported gradient.
Define
$$M(J[w])=-\int_EJ[w]\cdot D\phi dx ,\quad M(Q[w])=-\int_EQ[w]\cdot D\phi dx,\quad M(K[w])=-\int_EK[w]\cdot D\phi dx.$$
By Lemma \ref{lem:divQKJ},
$M(J[w]),$ $M(Q[w])$ and $M(K[w])$ is independent of $\phi$.

	\begin{lemma}[Inner mass formula]\label{lem:inner-mass-super}
		One has
		\begin{align}
			M(J[w])
			&=\alpha\beta c^kI_{k-1}
			-\frac{n-k+1}{k-1}c^{k+1}I_{k-2},\label{eq:mass-inner-super}\\
            M(Q[w])&=c^kI_{k-1}>0,\label{eq:Q-inner-super}\\
			M(K[w])&=0\label{eq:MKw}.
		\end{align}
		In particular, \cref{prop:boundary-curvature} implies
		\begin{equation}\label{eq:mass-nonnegative-super}
			M(J[w])\ge0.
		\end{equation}
	\end{lemma}
	
	\begin{proof}
At the boundary $\partial\Omega$, in an orthonormal principal 
frame $(e_1,...,e_{n-1},\nu)$,
\begin{align}
   &w_{\varepsilon}=\beta,\quad (w_{\varepsilon})_{e_i}=0,\quad (w_{\varepsilon})_{e_i\nu}=(g_{\varepsilon})_{e_i},~i=1,...,n-1,\notag\\ 
    &D^2w_{\varepsilon}=
   \begin{pmatrix}
        g_{\varepsilon}K_{\Sigma}& \nabla_{\Sigma} g_{\varepsilon}\\
        (\nabla_{\Sigma} g_{\varepsilon})^T& (w_{\varepsilon})_{\nu\nu}
    \end{pmatrix},
\end{align}
where $K_{\Sigma}=diag(\kappa_1,...,\kappa_{n-1})$, $\nabla_{\Sigma}g_{\varepsilon}=((g_{\varepsilon})_{e_1},...,(g_{\varepsilon})_{e_{n-1}}),$ $e_i$ is the tangent vector field on $\partial\Omega$ for $i=1,...,n-1.$
We have 
\begin{equation}\label{sigmaexpansion}
    F_{\varepsilon}=\sigma_k(D^2w_{\varepsilon})=g_{\varepsilon}^kH_k+(w_{\varepsilon})_{\nu\nu}g_{\varepsilon}^{k-1}H_{k-1}-T_{k-2}^{e_ie_j}(g_{\varepsilon}K_{\Sigma})(g_{\varepsilon})_{e_i}(g_{\varepsilon})_{e_j},
\end{equation}
then 
\begin{align}
    T_{k-1}^{e_i\nu}(D^2w_{\varepsilon})&=\frac{\partial \sigma_k(D^2w_{\varepsilon})}{\partial (w_{\varepsilon})_{e_i\nu}}=-T_{k-2}^{e_ie_j}(g_{\varepsilon}K_{\Sigma})(g_{\varepsilon})_{e_j},\label{Teinu}\\
     T_{k-1}^{\nu\nu}(D^2w_{\varepsilon})&=\frac{\partial \sigma_k(D^2w_{\varepsilon})}{\partial (w_{\varepsilon})_{\nu\nu}}=g_\varepsilon^{k-1}H_{k-1}.\label{Tnunu}
\end{align}
According to  \eqref{Tnunu}, we have
\begin{align}\label{Qwvarcdotnu}
    Q[w_{\varepsilon}]\cdot\nu&=T_{k-1}^{ij}(D^2w_{\varepsilon})(w_{\varepsilon})_i\nu_j =T_{k-1}^{e_{i}\nu}(D^2w_{\varepsilon})(w_{\varepsilon})_{e_i} 
    +T_{k-1}^{\nu\nu}(D^2w_{\varepsilon})(w_{\varepsilon})_{\nu}\notag\\
    &=T_{k-1}^{\nu\nu}(D^2w_{\varepsilon})(w_{\varepsilon})_{\nu}=g_{\varepsilon}^{k}H_{k-1},
\end{align}
then by \eqref{diverQvar} 
\begin{align}\label{MQwvar}
    M(Q[w_{\varepsilon}])
    &=-\int_{E}Q[w_{\varepsilon}]\cdot D\phi dx\notag\\
    &=\int _{\partial\Omega}Q[w_{\varepsilon}]\cdot \nu ds
    +\int_E \phi\diver Q[w_{\varepsilon}] dx\notag\\
&=\int_{\partial\Omega}g_{\varepsilon}^{k}H_{k-1} +k\int_E \phi F_{\varepsilon}.
\end{align}

Since $w_{\varepsilon}=\beta $, $(w_{\varepsilon})_{\nu}=g_{\varepsilon}$ on $\partial\Omega,$ then $(w_{\varepsilon})_{\nu e_i}=(g_{\varepsilon})_{e_{i}},$ and
\begin{align}\label{znu}
    z^{w_{\varepsilon}}_\nu&=(x\cdot Dw_{\varepsilon}-\alpha w_{\varepsilon})_\nu=D(x\cdot Dw_{\varepsilon})\cdot\nu-\alpha (w_{\varepsilon})_\nu\notag\\
    &=(w_{\varepsilon})_\nu+x\cdot (Dw_{\varepsilon})_{\nu}-\alpha(w_{\varepsilon})_\nu
    =q(w_{\varepsilon})_\nu+x\cdot\nu(w_{\varepsilon})_{\nu\nu}+x\cdot e_i(w_{\varepsilon})_{e_i \nu}\notag\\
    &=qg_{\varepsilon}+h_{\Omega}(w_{\varepsilon})_{\nu\nu}+x^{\tau}\cdot \nabla_{\Sigma}g_{\varepsilon}.
\end{align}
\begin{align}\label{zei}
z_{e_i}^{w_{\varepsilon}}
    &=(x\cdot Dw_{\varepsilon}-\alpha w_{\varepsilon})_{e_i}=q(w_{\varepsilon})_{e_i}+x_l(w_{\varepsilon})_{le_i}\notag\\
    &=h_{\Omega}(w_{\varepsilon})_{\nu e_i}+x_{e_j}(w_{\varepsilon})_{e_ie_j}
    =h_{\Omega}(g_{\varepsilon})_{e_i}+g_{\varepsilon}\kappa_ix_{e_i}.
\end{align}
Combining \eqref{Teinu}, \eqref{Tnunu}, \eqref{znu}, and
\eqref{zei}, and writing
\[
 s_\varepsilon=(w_\varepsilon)_{\nu\nu},
 \qquad
 P_\varepsilon=T_{k-2}(g_\varepsilon K_\Sigma),
\]
we obtain
\begin{align}
 K[w_\varepsilon]\cdot\nu
={}&T_{k-1}^{e_i\nu}(D^2w_\varepsilon)
       Z^{w_\varepsilon}_{e_i}
   +T_{k-1}^{\nu\nu}(D^2w_\varepsilon)
       Z^{w_\varepsilon}_{\nu}\notag\\
={}&-P_\varepsilon^{ij}(g_\varepsilon)_{e_j}
 \left(
 h_\Omega(g_\varepsilon)_{e_i}
 +g_\varepsilon(K_\Sigma x^\tau)_{e_i}
 \right)\notag\\
&+g_\varepsilon^{k-1}H_{k-1}
 \left(
 qg_\varepsilon+h_\Omega s_\varepsilon
 +x^\tau\cdot\nabla_\Sigma g_\varepsilon
 \right).\label{eq:K-boundary-expand-super}
\end{align}
Here and below repeated tangential indices are summed.

By \eqref{sigmaexpansion}, there is
\[
 s_\varepsilon g_\varepsilon^{k-1}H_{k-1}
 -P_\varepsilon^{ij}
  (g_\varepsilon)_{e_i}(g_\varepsilon)_{e_j}
 =F_\varepsilon-g_\varepsilon^kH_k.
\]
Moreover,
\[
 P_\varepsilon
 =g_\varepsilon^{k-2}T_{k-2}(K_\Sigma),
\]
and the Newton identity
\[
 T_{k-1}(K_\Sigma)
 =H_{k-1}I-T_{k-2}(K_\Sigma)K_\Sigma
\]
gives
\begin{align*}
&g_\varepsilon^{k-1}H_{k-1}
 x^\tau\cdot\nabla_\Sigma g_\varepsilon
 -g_\varepsilon
 P_\varepsilon^{ij}(g_\varepsilon)_{e_j}
 (K_\Sigma x^\tau)_{e_i}\\
&\qquad
=g_\varepsilon^{k-1}
 T_{k-1}(K_\Sigma)x^\tau
 \cdot\nabla_\Sigma g_\varepsilon.
\end{align*}
Consequently,
\begin{align}
 K[w_\varepsilon]\cdot\nu
={}&qg_\varepsilon^kH_{k-1}
 +h_\Omega F_\varepsilon
 -h_\Omega g_\varepsilon^kH_k\notag\\
&+g_\varepsilon^{k-1}
 T_{k-1}(K_\Sigma)x^\tau
 \cdot\nabla_\Sigma g_\varepsilon.
\label{eq:K-boundary-super}
\end{align}

We next integrate the last term. Set
\[
 Y=T_{k-1}(K_\Sigma)x^\tau.
\]
Since the hypersurface Newton tensor is divergence-free,
\[
 \operatorname{div}_\Sigma T_{k-1}(K_\Sigma)=0.
\]
Also,
\[
 \nabla_\Sigma x^\tau=I-h_\Omega K_\Sigma.
\]
Therefore,
\begin{align}
 \operatorname{div}_\Sigma Y
 &=\operatorname{tr}T_{k-1}(K_\Sigma)
 -h_\Omega
 T_{k-1}(K_\Sigma):K_\Sigma\notag\\
 &=(n-k)H_{k-1}-kh_\Omega H_k.
\label{eq:surface-div-super}
\end{align}
Because $\partial\Omega$ is closed,
\[
 \int_{\partial\Omega}
 \operatorname{div}_\Sigma(g_\varepsilon^kY)\,ds=0.
\]
Hence
\begin{align}
&\int_{\partial\Omega}
 g_\varepsilon^{k-1}
 T_{k-1}(K_\Sigma)x^\tau
 \cdot\nabla_\Sigma g_\varepsilon\,ds\notag\\
&\quad
=\frac1k\int_{\partial\Omega}
 Y\cdot\nabla_\Sigma(g_\varepsilon^k)\,ds\notag\\
&\quad
=-\frac1k\int_{\partial\Omega}
 g_\varepsilon^k\operatorname{div}_\Sigma Y\,ds\notag\\
&\quad
=-q\int_{\partial\Omega}
 g_\varepsilon^kH_{k-1}\,ds
 +\int_{\partial\Omega}
 h_\Omega g_\varepsilon^kH_k\,ds,
\label{eq:surface-ibp-super}
\end{align}
where $q=(n-k)/k$. Integrating
\eqref{eq:K-boundary-super} and using
\eqref{eq:surface-ibp-super}, the first and last curvature
terms cancel, and we obtain
\begin{equation}
 \int_{\partial\Omega}
 K[w_\varepsilon]\cdot\nu\,ds
 =
 \int_{\partial\Omega}h_\Omega F_\varepsilon\,ds.
\label{eq:K-integrated-super}
\end{equation}
Combining \eqref{diverKvar} and \eqref{eq:K-integrated-super}, we have 
\begin{align}
 M(K[w_\varepsilon])
={}&-\int_EK[w_\varepsilon]\cdot D\phi\,dx\notag\\
={}&\int_{\partial\Omega}
 K[w_\varepsilon]\cdot\nu\,ds
 +\int_E\phi\,\operatorname{div}K[w_\varepsilon]\,dx\notag\\
={}&\int_{\partial\Omega}h_\Omega F_\varepsilon\,ds\notag\\
&+\int_E\phi
 \left(
 x\cdot DF_\varepsilon+k(1+q)F_\varepsilon
 \right)\,dx.
\label{eq:K-mass-eps-super}
\end{align}

We now turn to the Jacobi current. Since
\[
 J[w_\varepsilon]
 =w_\varepsilon K[w_\varepsilon]
 -Z^{w_\varepsilon}Q[w_\varepsilon],
\]
then by \eqref{Qwvarcdotnu} and \eqref{eq:K-integrated-super},
\begin{align}
\int_{\partial\Omega}
 J[w_\varepsilon]\cdot\nu\,ds
={}&\alpha\beta
 \int_{\partial\Omega}
 g_\varepsilon^kH_{k-1}\,ds-\int_{\partial\Omega}
 h_\Omega g_\varepsilon^{k+1}H_{k-1}\,ds
 +\beta\int_{\partial\Omega}
 h_\Omega F_\varepsilon\,ds,
\label{eq:J-integrated-super}
\end{align}
combining with \eqref{diverJvar}, we obtain
\begin{align}\label{eq:J-mass-eps-super}
 M(J[w_\varepsilon])
={}&\int_{\partial\Omega}
 J[w_\varepsilon]\cdot\nu\,ds
 +\int_E\phi
 \left[
 w_\varepsilon
 \left(
 x\cdot DF_\varepsilon+k(1+q)F_\varepsilon
 \right)
 -kZ^{w_\varepsilon}F_\varepsilon
 \right]dx\notag\\
 =&\alpha\beta
 \int_{\partial\Omega}
 g_\varepsilon^kH_{k-1}\,ds-\int_{\partial\Omega}
 h_\Omega g_\varepsilon^{k+1}H_{k-1}\,ds
 +\beta\int_{\partial\Omega}
 h_\Omega F_\varepsilon\,ds\notag\\
 &+\int_E\phi
 \left[
 w_\varepsilon
 \left(
 x\cdot DF_\varepsilon+k(1+q)F_\varepsilon
 \right)
 -kZ^{w_\varepsilon}F_\varepsilon
 \right]dx.
\end{align}

We now pass to the limit. By the approximation properties,
\[
 w_\varepsilon\longrightarrow w
 \quad\text{in }C^1_{\mathrm{loc}}(\overline E),
 \qquad
 g_\varepsilon\longrightarrow c
 \quad\text{uniformly on }\partial\Omega,
\]
the functions $w_\varepsilon$ are uniformly bounded in
$C^{1,1}_{\mathrm{loc}}(\overline E)$, and
\[
 F_\varepsilon\longrightarrow0
 \quad\text{in }C^1
\]
on every fixed compact set containing $\operatorname{supp}\phi$ and
$\partial\Omega$. Thus every volume error in the formulas above
converges to zero. The weak continuity of the Newton--Jacobi
currents also gives, for
$\mathcal P=Q,K,J$,
\[
 M(\mathcal P[w_\varepsilon])
 =-\int_E\mathcal P[w_\varepsilon]\cdot D\phi\,dx
 \longrightarrow
 -\int_E\mathcal P[w]\cdot D\phi\,dx
 =M(\mathcal P[w]).
\]
From \eqref{MQwvar}, \eqref{eq:K-mass-eps-super} and \eqref{eq:J-mass-eps-super},
we conclude that
\begin{equation}\label{MQw}
     M(Q[w])
 =c^k\int_{\partial\Omega}H_{k-1}\,dS
 =c^kI_{k-1}>0,
\end{equation}
\begin{equation}\label{MKw}
     M(K[w])=0,
\end{equation}
and 
\begin{equation}\label{MJw}
 M(J[w])
 =\alpha\beta c^kI_{k-1}
 -c^{k+1}
 \int_{\partial\Omega}
 h_\Omega H_{k-1}\,dS.
\end{equation}
The Hsiung–Minkowski formula \eqref{eq:minkowski} gives
\[
 \int_{\partial\Omega}
 h_\Omega H_{k-1}\,dS
 =
 \frac{n-k+1}{k-1}I_{k-2}.
\]
Therefore,
\[
 M(J[w])
 =
 \alpha\beta c^kI_{k-1}
 -\frac{n-k+1}{k-1}
 c^{k+1}I_{k-2},
\]
where
\[
 c=\alpha\beta b.
\]

It remains to prove that 
$M(J[w])\geq 0.$
 By Proposition
\ref{prop:boundary-curvature},
\[
 H_k\ge qbH_{k-1}
 \qquad\text{on }\partial\Omega,
\]
Multiplying by $h_\Omega>0$ and integrating, the
Hsiung--Minkowski formulas give
\[
 qI_{k-1}
 =\int_{\partial\Omega}h_\Omega H_k\,dS
 \ge qb\int_{\partial\Omega}
 h_\Omega H_{k-1}\,dS
 =
 qb\frac{n-k+1}{k-1}I_{k-2}.
\]
Since $q>0$,
\[
 I_{k-1}
 \ge
 b\frac{n-k+1}{k-1}I_{k-2}.
\]
Using $c=\alpha\beta b$, we finally obtain
\begin{align*}
 M(J[w])
 &=c^k\left(
 \alpha\beta I_{k-1}
 -\frac{n-k+1}{k-1}cI_{k-2}
 \right)\geq 0.
\end{align*}
\end{proof}

	\medskip
	\noindent\textbf{Step 3: evaluation of the mass at infinity.}
Let
\[
 \cA_*=\{x\in\mathbb R^n:1/4<|x|<4\}.
\]
Choose a smooth radial cutoff $\eta$ which equals one near $|x|=1/2$ and vanishes near $|x|=2$.  For $U\in C^{1,1}(\cA_*)$, define almost everywhere
\[
 Z^U=x\cdot DU-\alpha U,
\]
\[
 J^j[U]=T_{k-1}^{ij}(D^2U)\bigl(U(Z^U)_i-Z^UU_i\bigr),
\]
and
\begin{equation}\label{eq:M-eta}
 \mathscr M_\eta[U]=-\int_{\cA_*}J[U]\cdot D\eta\,dx.
\end{equation}
When $\sigma_k(D^2U)=0$, this is the cutoff definition of the weak
Newton--Jacobi flux.  If $U$ is smooth, it also equals the classical spherical
flux.

To evaluate the mass at infinity we use the following stability fact.
\begin{lemma}[Quantitative current stability]\label{lem:current-stability}
Let $U,V\in C^{1,1}(\cA_*)$ satisfy, almost everywhere,
\[
 D^2U,D^2V\in\overline\Gamma_k,
 \qquad
 \sigma_k(D^2U)=\sigma_k(D^2V)=0,
\]
and
\[
 \|U\|_{C^{1,1}(\cA_*)}+\|V\|_{C^{1,1}(\cA_*)}\le M.
\]
Assume that
\begin{equation}\label{eq:average-T}
 \overline T^{ij}
 =\int_0^1T_{k-1}^{ij}\bigl(D^2(V+t(U-V))\bigr)\,dt
\end{equation}
satisfies
\[
 \lambda I\le\overline T\le\Lambda I
\]
on a neighborhood of $\operatorname{supp}D\eta$.  Then
\begin{equation}\label{eq:current-stability}
 |\mathscr M_\eta[U]-\mathscr M_\eta[V]|
 \le C\|U-V\|_{L^\infty(\cA_*)},
\end{equation}
where $C$ depends only on $n,k,M,\lambda,\Lambda$, and $\eta$.
\end{lemma}

\begin{proof}
Set $H=U-V$.  By the fundamental theorem of calculus,
\[
 \overline T^{ij}H_{ij}=0.
\]
Each Newton tensor in \eqref{eq:average-T} is divergence-free in the sense of
distributions, hence $\partial_i\overline T^{ij}=0$ and
\[
 \partial_i(\overline T^{ij}H_j)=0.
\]
Let $\mathcal O\Subset\cA_*$ be a neighborhood of
$\operatorname{supp}D\eta$ on which the ellipticity assumption holds, and
choose $\zeta\in C_c^\infty(\mathcal O)$ equal to one near
$\operatorname{supp}D\eta$.  Testing the last equation with $\zeta^2H$ gives
\begin{align*}
 \lambda\int\zeta^2|DH|^2
 &\le -2\int \zeta H\,\overline T^{ij}H_j\zeta_i\\
 &\le \frac\lambda2\int\zeta^2|DH|^2
      +C\int H^2|D\zeta|^2.
\end{align*}
Consequently,
\begin{equation}\label{eq:Caccioppoli}
 \|DH\|_{L^2(\operatorname{supp}D\eta)}
 \le C\|H\|_{L^\infty(\cA_*)}.
\end{equation}

We next prove a purely algebraic estimate.  No equation for $U$ or $V$ is
used in this part.  First suppose that they are smooth.  Write
\[
 B_i(U)=U(Z_U)_i-Z_UU_i.
\]
Since $q+\alpha=1$,
\begin{equation}\label{eq:B-formula}
 B_i(U)=Ux_\ell U_{i\ell}+UU_i-(x_\ell U_\ell)U_i.
\end{equation}
Then
\begin{align}
 J_U-J_V
 ={}&T_{k-1}(D^2V)\bigl(B(U)-B(V)\bigr)\notag\\
 &+\bigl(T_{k-1}(D^2U)-T_{k-1}(D^2V)\bigr)B(U).
\label{eq:J-difference}
\end{align}
In the first term,
\begin{align}
 B_i(U)-B_i(V)
 ={}&Hx_\ell V_{i\ell}+Ux_\ell H_{i\ell}+HV_i+UH_i\notag\\
 &-(x_\ell V_\ell)H_i-(x_\ell H_\ell)U_i.
\label{eq:B-difference}
\end{align}
The only second derivative of $H$ occurs in $Ux_\ell H_{i\ell}$.  In its
contribution to the flux difference, integration in the $i$-index gives
\begin{equation}\label{eq:first-Hessian-integration}
 \int\eta_jT_{k-1}^{ij}(D^2V)Ux_\ell H_{i\ell}
 =-\int H_\ell T_{k-1}^{ij}(D^2V)
 \left(\eta_{ji}Ux_\ell
 +\eta_jU_ix_\ell+\eta_jU\delta_{i\ell}\right).
\end{equation}
Here we used
\[
 \partial_iT_{k-1}^{ij}(D^2V)=0
\]
and the compact support of $D\eta$.  All other terms in
\eqref{eq:B-difference} contain only $H$ or $DH$.  Thus the first line of
\eqref{eq:J-difference}, after pairing with $D\eta$, is bounded by
\begin{equation}\label{eq:first-current-bound}
 C\int_{\operatorname{supp}D\eta}(|H|+|DH|).
\end{equation}

For the second term in \eqref{eq:J-difference}, write
\begin{equation}\label{eq:T-difference}
 T_{k-1}^{ij}(D^2U)-T_{k-1}^{ij}(D^2V)
 =\mathbb S^{ijab}H_{ab},
\end{equation}
where
\begin{equation}\label{eq:S-tensor}
 \mathbb S^{ijab}
 =\int_0^1
 \frac{\partial T_{k-1}^{ij}}{\partial r_{ab}}
 \bigl(D^2(V+tH)\bigr)\,dt.
\end{equation}
The nonsymmetrized derivative may be chosen so that
\begin{equation}\label{eq:S-generalized-delta}
 \mathbb S^{ijab}
 =\frac1{(k-2)!}\int_0^1
 \delta^{i a i_3\cdots i_k}_{j b j_3\cdots j_k}
 (V+tH)_{i_3j_3}\cdots(V+tH)_{i_kj_k}\,dt.
\end{equation}
Because $H_{ab}=H_{ba}$, this convention gives exactly
\eqref{eq:T-difference}.  Antisymmetry of the generalized delta, together
with symmetry of third derivatives, gives
\begin{equation}\label{eq:S-properties}
 \partial_a\mathbb S^{ijab}=0,
 \qquad
 \mathbb S^{ijab}=-\mathbb S^{ajib}.
\end{equation}
Indeed, in $\partial_a\mathbb S^{ijab}$ every differentiated Hessian is
symmetric in $a$ and one of $i_3,\ldots,i_k$, whereas the delta is
antisymmetric in the same two upper indices.

Pairing the second line of \eqref{eq:J-difference} with $D\eta$ and
integrating once in the $a$-index yields
\begin{align}
 &\int\eta_j\mathbb S^{ijab}H_{ab}B_i(U)\notag\\
 &\quad=-\int H_b\left(
 \eta_{ja}\mathbb S^{ijab}B_i(U)
 +\eta_j\mathbb S^{ijab}\partial_aB_i(U)
 \right),
 \label{eq:second-Hessian-integration}
\end{align}
where the divergence term in \eqref{eq:S-properties} has disappeared.
Differentiating \eqref{eq:B-formula},
\begin{align*}
 \partial_aB_i(U)
 ={}&U_ax_\ell U_{i\ell}+Ux_\ell U_{i\ell a}+2UU_{ia}\\
 &-x_\ell U_{\ell a}U_i-(x_\ell U_\ell)U_{ia}.
\end{align*}
The only third derivative is $U_{i\ell a}$, symmetric in $i,a$, and it vanishes upon contraction with the antisymmetric tensor in \eqref{eq:S-properties}.  All remaining coefficients are bounded by the assumed $C^{1,1}$ norm.  Hence
\[
 |\mathscr M_\eta[U]-\mathscr M_\eta[V]|
 \le C\int_{\operatorname{supp}D\eta}(|H|+|DH|).
\]
Equation \eqref{eq:Caccioppoli} and Cauchy--Schwarz prove
\eqref{eq:current-stability} in the smooth case.

It remains to justify that the algebraic estimate survives at $C^{1,1}$
regularity.  Extend $U,V$ from a fixed neighborhood of
$\operatorname{supp}D\eta$ and mollify them there, obtaining $U_\delta,V_\delta$.
Their $C^{1,1}$ norms stay uniformly bounded.  Apply
\eqref{eq:first-Hessian-integration} and
\eqref{eq:second-Hessian-integration} to this smooth pair.  The third
derivatives cancel by \eqref{eq:S-properties} before any estimate is taken, so
the resulting constant depends only on the uniform $C^{1,1}$ bound and $\eta$.
Thus
\begin{equation}\label{eq:mollified-algebra-bound}
 |\mathscr M_\eta[U_\delta]-\mathscr M_\eta[V_\delta]|
 \le C\int_{\operatorname{supp}D\eta}
 (|H_\delta|+|DH_\delta|).
\end{equation}
For this fixed pair of $C^{1,1}$ functions,
\[
 U_\delta\to U,\quad V_\delta\to V\quad\hbox{in }C^1,
 \qquad
 D^2U_\delta\to D^2U,\quad D^2V_\delta\to D^2V
 \quad\hbox{in }L^p
\]
for every finite $p$.  Hence the currents in
\eqref{eq:mollified-algebra-bound} converge in $L^1$, while
$H_\delta\to H$ in $C^1$ on the support of $D\eta$.  Passing to the limit gives
the same algebraic estimate for $U,V$.  Notice that the mollified functions
are not required to satisfy the Hessian equation: the equation was used only
for the unmollified pair, in deriving the Caccioppoli estimate
\eqref{eq:Caccioppoli}.  Combining that estimate with the limiting algebraic
bound proves \eqref{eq:current-stability} for $C^{1,1}$ functions.
\end{proof}

%	With the constant $C_0$ fixed in
%	\eqref{eq:super-finite-part-step},
	%\begin{equation}\label{eq:w-finite-super}
	%	w(x)=\mathfrak a|x|^\alpha+C_0+o(1).
	%\end{equation}

\begin{lemma}[Mass equals the finite part]\label{lem:mass-finite-part}
For the function $w$ in \eqref{eq:w-super},
\begin{equation}\label{eq:mass-C0}
 M(J[w])=\alpha C_0M(Q[w]).
\end{equation}
\end{lemma}

\begin{proof}
Set
\[
 \widehat w=w-C_0.
\]
Then
\[
 \widehat w(x)-\mathfrak a|x|^\alpha=o(1).
\]
For $R\to\infty$, define on $\cA_*$,
\[
 U_R(y)=R^{-\alpha}\widehat w(Ry),
 \qquad
 \Phi(y)=\mathfrak a|y|^\alpha.
\]
Then
\begin{equation}\label{eq:UR-Phi-sup}
 \|U_R-\Phi\|_{L^\infty(\cA_*)}=o(R^{-\alpha}),
\end{equation}
and Proposition \ref{prop:star-package} gives uniform $C^{1,1}$ bounds.  Moreover, $T_{k-1}(D^2\Phi)>0$.  The same short-initial-segment argument used in \eqref{eq:linear-uniform} shows that the average linearization between $U_R$ and $\Phi$ is uniformly elliptic on the fixed annulus.  Applying \cref{lem:current-stability},
\begin{equation}\label{eq:mass-UR-small}
 |\mathscr M_\eta[U_R]-\mathscr M_\eta[\Phi]|
 =o(R^{-\alpha}).
\end{equation}
Since $\Phi$ is $\alpha$-homogeneous,
\[
 Z^\Phi=x\cdot D\Phi-\alpha\Phi=0,
\]
then 
\begin{equation}\label{MPhi}
    \mathscr M_\eta[\Phi]=0.
\end{equation}

\begin{align}
    \mathscr M_\eta[U_R]&=-\int_{\cA_*}J[U_R]\cdot D\eta dx\notag\\
    &=-\int_{\cA_*}[T_{k-1}(D^2U_R)(U_RDZ^{U_R}-Z^{U_R}DU_R)]\cdot D\eta dx\notag\\
    &=-\int_{\cA_*}R^{(2-\alpha)(k-1)+1-2\alpha}[T_{k-1}(D^2\widehat{w})(\widehat{w}DZ^{\hat{w}}-Z^{\widehat{w}}D\widehat{w})](Rx)\cdot D\eta dx\notag\\
    &=-\int_{B_{4R}\setminus B_{\frac{1}{4}R}}R^{(2-\alpha)(k-1)+1-2\alpha+1-n}[T_{k-1}(D^2\widehat{w})(\widehat{w}DZ^{\hat{w}}-Z^{\widehat{w}}D\widehat{w})](y)\cdot D_y\eta(\frac{y}{R})dy\notag\\
    &=-R^{-\alpha}\int_{B_{4R}\setminus B_{\frac{1}{4}R}}J([\widehat{w}])(y)\cdot D_y\eta(\frac{y}{R})dy\notag\\
    &=R^{-\alpha}M(J[\widehat{w}]).\notag
\end{align}
Here the last equality is the cutoff-independence from Lemma
\ref{lem:divQKJ}.
Together with \eqref{eq:mass-UR-small} and \eqref{MPhi}, this yields
\[
 R^{-\alpha}|M(J[\widehat w])|=o(R^{-\alpha}),
\]
so
\begin{equation}\label{eq:mass-hat-zero}
 M(J[\widehat w])=0.
\end{equation}

It remains to compute the effect of a constant shift.  Since
\[
 Z^{w-C_0}=Z^w+\alpha C_0,
\]
one has
\[
  M(J[w-C_0])
 =M(J[w])
 -C_0M( K[w])
 -\alpha C_0M(Q[w]).
\]
By \eqref{eq:MKw}, \eqref{eq:mass-hat-zero}, $M(K[w])=M(J[w-C_0])=0$.  Therefore
\[
 M(J[w])=\alpha C_0M(Q[w]).
\]
\end{proof}

	\medskip
	\noindent\textbf{Step 4: the sign of the finite part and the boundary ratio.}
	
	\begin{proposition}[Supercritical boundary ratio]\label{prop:ratio-super}
		One has
		\begin{equation}\label{eq:universal-ratio-super}
			I_{k-1}=b\frac{n-k+1}{k-1}I_{k-2}.
		\end{equation}
	\end{proposition}
	
	\begin{proof}
		Define
		\[
		G=\left(\frac w{\mathfrak a}\right)^{1/\alpha}=R_\ast v.
		\]
		By Lemma \ref{thm:gradient}, $|DG|\le1$.  Write the star-shaped boundary as
		$\{\rho(\theta)\theta\}$.  Along every exterior ray,
		\begin{equation}\label{eq:G-upper-super}
			G(r\theta)\le R_\ast+r-\rho(\theta)\le r+C.
		\end{equation}
		On the other hand, \eqref{eq:super-finite-part-step} gives
		\begin{equation}\label{eq:G-expansion-super}
			G(x)=|x|+\frac{C_0}{\mathfrak a\alpha}|x|^q+o(|x|^q).
		\end{equation}
		If $C_0>0$, this contradicts \eqref{eq:G-upper-super}; hence $C_0\le0$.
		By Lemma \ref{lem:mass-finite-part}, $M(J[w])\le0$, while
		\eqref{eq:mass-nonnegative-super} gives $M(J[w])\ge0$.  Therefore
		$C_0=M(J[w])=0$.  Substituting this into
		\eqref{eq:mass-inner-super} and using $c=\alpha\beta b$ gives
		\eqref{eq:universal-ratio-super}.
	\end{proof}
	
	\medskip
	\noindent\textbf{Step 5: geometric rigidity and identification of the solution.}
	
	The common boundary inequality proved in Proposition \ref{prop:boundary-curvature} gives
	$H_k\ge qbH_{k-1}>0$, so strict $(k-1)$-convexity improves to strict
	$k$-convexity.  The integral identity in Proposition \ref{prop:ratio-super} is
	exactly \eqref{eq:closure-integral}.  Hence Proposition 
	\ref{prop:geometric-closure} yields
	\[
	\Omega=B_R(x_0),\qquad
	R=b^{-1}=R_\ast
	=\left(\frac{\mathfrak a(1-q)}c\right)^{1/q}
	=\left(\frac{\mathfrak a\alpha}c\right)^{1/q}.
	\]
	The radial function
	\[
	u_0(x)=\mathfrak a|x-x_0|^\alpha+1-\mathfrak aR^\alpha
	\]
	has the same boundary value and the same leading coefficient at
	infinity and solves $\sigma_k(D^2u_0)=0$.  Uniqueness in the prescribed
	fundamental-growth class gives $u=u_0$.
	\hfill$\Box$

\textbf{Acknowledgements.} The author would like to thank Professor Xinan Ma and Jiahuan Li for their helpful conversations on this
work.
The author acknowledges the use of AI tools. All mathematical statements and proofs were independently verified by the author, who takes full responsibility for the content of the manuscript.

	\bigskip
	
	\noindent

	\textsc{Zhihui Zhang}\\
	School of Mathematics and Statistics, Beijing Institute of Technology, \\
	Beijing, 100081, People's Republic of China\\
	Email: \href{mailto:zzhwisdom@zjnu.edu.cn}{zzhwisdom@zjnu.edu.cn}
\end{document}